\documentclass[letterpaper]{article} % DO NOT CHANGE THIS
\usepackage[submission]{arxiv}  % DO NOT CHANGE THIS
\usepackage{times}  % DO NOT CHANGE THIS
\usepackage{helvet}  % DO NOT CHANGE THIS
\usepackage{courier}  % DO NOT CHANGE THIS
\usepackage[hyphens]{url}  % DO NOT CHANGE THIS
\usepackage{graphicx} % DO NOT CHANGE THIS
\usepackage{natbib}  % DO NOT CHANGE THIS AND DO NOT ADD ANY OPTIONS TO IT
\usepackage{caption} % DO NOT CHANGE THIS AND DO NOT ADD ANY OPTIONS TO IT
\usepackage{algorithm}
\usepackage{algorithmic}

\usepackage{amsmath}
\usepackage{bm}
\usepackage{amssymb}
\usepackage{mathtools}
\usepackage{amsthm}
\newtheorem{theorem}{Theorem}
\newtheorem{lemma}{Lemma}

\usepackage{multirow}
\usepackage{booktabs} % for professional tables
\usepackage{subfig}

\usepackage{newfloat}
\usepackage{listings}
\DeclareCaptionStyle{ruled}{labelfont=normalfont,labelsep=colon,strut=off} % DO NOT CHANGE THIS
\floatstyle{ruled}
\newfloat{listing}{tb}{lst}{}
\floatname{listing}{Listing}
\makeatletter
\def\showauthors@on{T}
\makeatother
\title{Physics-Informed Tailored Finite Point Operator Network for Parametric Interface Problems}
\author{
    Ting Du\textsuperscript{\rm 1},
    Xianliang Xu\textsuperscript{\rm 1},
    Wang Kong\textsuperscript{\rm 2},
    Ye Li\textsuperscript{\rm 2}\thanks{Corresponding author},
    Zhongyi Huang\textsuperscript{\rm 1}\footnotemark[1]
}
\affiliations{
    \textsuperscript{\rm 1}Tsinghua University\\
    \textsuperscript{\rm 2}Nanjing University of Aeronautics and Astronautics\\
    dt20@mails.tsinghua.edu.cn, xuxl19@mails.tsinghua.edu.cn, wkong@nuaa.edu.cn, yeli20@nuaa.edu.cn, zhongyih@tsinghua.edu.cn

}

\usepackage{bibentry}
\begin{document}

\maketitle

\begin{abstract}
Learning operators for parametric partial differential equations (PDEs) using neural networks has gained significant attention in recent years. However, standard approaches like Deep Operator Networks (DeepONets) require extensive labeled data, and physics-informed DeepONets encounter training challenges. 
In this paper, we introduce a novel \textbf{p}hysics-\textbf{i}nformed \textbf{t}ailored \textbf{f}inite \textbf{p}oint \textbf{o}perator \textbf{net}work (PI-TFPONet) method to solve parametric interface problems without the need for labeled data. Our method fully leverages the prior physical information of the problem, eliminating the need to include the PDE residual in the loss function, thereby avoiding training challenges. The PI-TFPONet is specifically designed to address certain properties of the problem, allowing us to naturally obtain an approximate solution that closely matches the exact solution. Our method is theoretically proven to converge if the local mesh size is sufficiently small and the training loss is minimized. Notably, our approach is uniformly convergent for singularly perturbed interface problems.
Extensive numerical studies show that our unsupervised PI-TFPONet is comparable to or outperforms existing state-of-the-art supervised deep operator networks in terms of accuracy and versatility.

\end{abstract}

\section{Introduction}

Solving parametric partial differential equations (PDEs) through deep learning has attracted extensive attention recently. Thanks to the universal approximation theorem for operators \cite{chen1995universal,lu2021learning}, neural networks can approximate solutions to PDEs. Several operator neural networks, such as the deep operator network (DeepONet) \cite{lu2021learning}, the Fourier neural operator (FNO) \cite{li2020fourier}, and their variants, have been developed to solve parametric PDEs. The physics-informed versions of these methods offer the advantage of requiring no additional labeled data. However, they can encounter training failures when using physics-informed loss functions \cite{wang2022PINNfail,wang2022improved}. Most successful applications focus on PDEs with smooth solutions, while less attention has been given to problems with non-smooth solutions or piece-wise smooth solutions. A typical example is the elliptic interface problem, where the solution and its derivatives exhibit jump discontinuities across the interface, as shown in Fig.~\ref{fig_interface_domain}.

\begin{figure}[h]
	\centering
	\includegraphics[width=0.8\columnwidth]{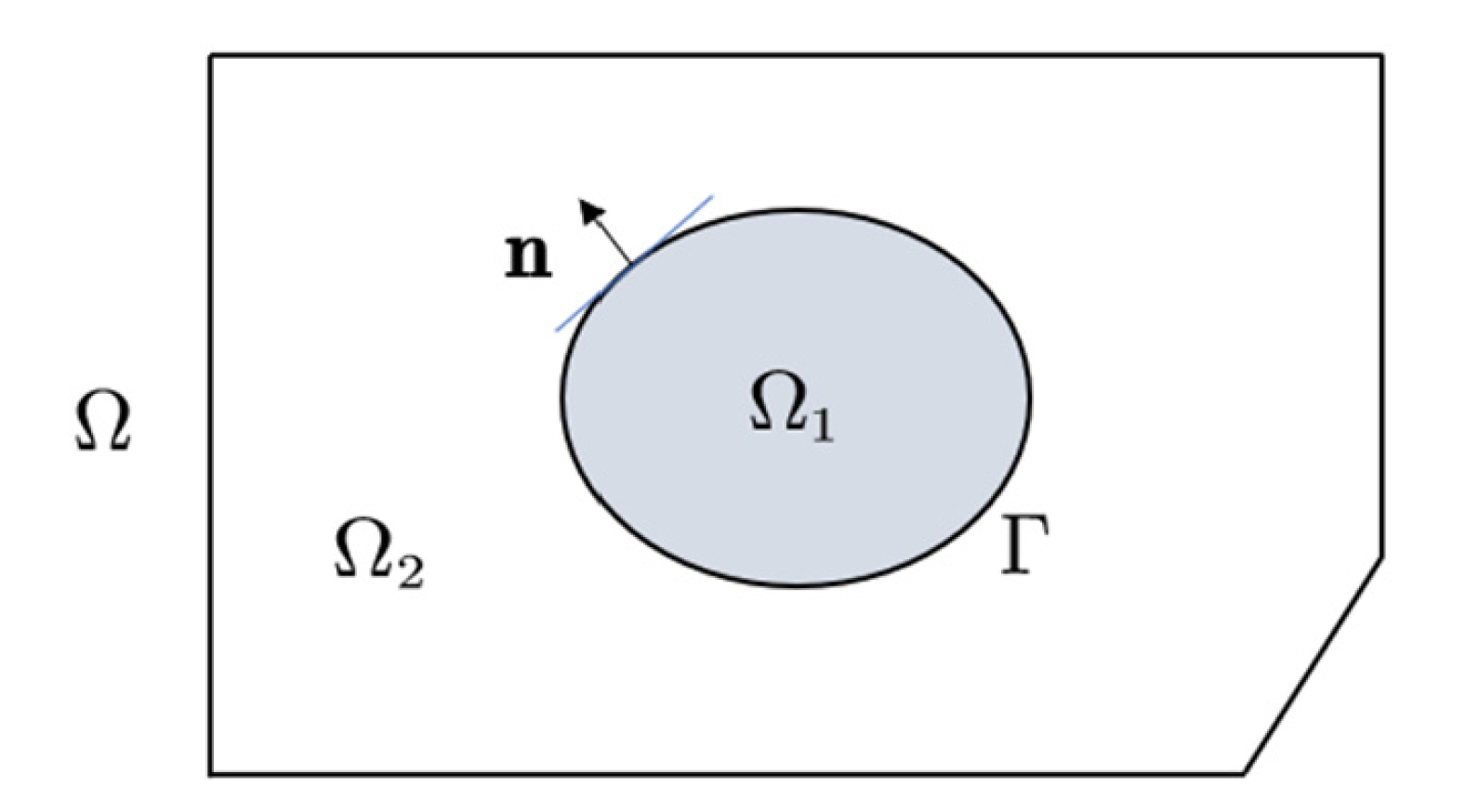}
	\caption{A sketch of the domain $\Omega$ and interface $\Gamma$ from \cite{wu2024solving}. Here $\Gamma$ divide $\Omega$ into two disjoint subdomains $\Omega_1,\Omega_2$.}
	\label{fig_interface_domain}
\end{figure}

Elliptic interface problems have widespread applications across various fields, including fluid mechanics \cite{fadlun2000combined,sussman1999efficient}, materials science \cite{LIU2020109017,WANG2019117}, electromagnetics \cite{hesthaven2003high}, and biomimetics \cite{ji2018finite}. 
The rapid and accurate simulation of these differential equations are critical in both scientific research and engineering applications \cite{azizzadenesheli2024neural}.
One difficulty of the interface problems is that the low global regularity of the solution and the irregular geometry of the interface bring additional challenges, manifesting singularities at interfaces \cite{babuvska1970finite,huang2009tailored,kellogg1971singularities}. 
The variable nature of singularities and the complex geometries at interfaces render standard neural operators less effective for achieving precision.
Although there are some studies like IONet \cite{wu2024solving} using different DeepONets for interface problems on different subdomains, research in solving parametric interface problems is rather limited compared to the fruitful works in the application of neural operators.

\subsection{Contributions}

In this paper, we propose a novel \textbf{p}hysics-\textbf{i}nformed \textbf{t}ailored \textbf{f}inite \textbf{p}oint \textbf{o}perator \textbf{net}work (PI-TFPONet) method to solve parametric interface problems. Our technique leverages prior local physical information about the interface problem \cite{huang2009tailored,wu2024solving}. Mathematical analysis shows that the solution in a small local region can be effectively approximated by a linear combination of local basis functions. Instead of learning the solution directly, we learn the coefficients of these local basis functions, then reconstruct the solution naturally. This approach avoids the training difficulties encountered in physics-informed DeepONets. Using this prior physical information, we can train the model and reconstruct the neural network solution without any additional training data. Below is a summary of our paper's contributions:

\begin{itemize}
\item We propose a novel unsupervised PI-TFPONet method to solve parametric interface problems. Our models can be trained on coarse grids while delivering accurate predictions for various functions and finer locations.
\item We provide a theoretical error estimation for our PI-TFPONet method, showing that the error converges to zero as long as the collocation points are sufficient and the training loss is minimized.
\item We demonstrate the accuracy and efficiency of our method through extensive numerical results. Our models achieve comparable or superior accuracy to supervised models without requiring large amounts of labeled training data, multiple networks, or the training difficulties faced by existing unsupervised methods.
\end{itemize}

\subsection{Related Works}
\paragraph{Data-based PDE solver.}
The standard deep operator network (DeepONet) \cite{lu2021learning}, Fourier neural operator (FNO) \cite{li2020fourier} are both data-based supervised methods, with both theoretical convergence guarantees \cite{deng2022approximation,kovachki2021universal,lanthaler2022error} and  diverse applications \cite{shukla2024deep,haghighat2024deeponet,lu2022comprehensive,li2024geometry}.
However, they often demands an extensive dataset to enhance predictive performance.

\paragraph{Physics-informed PDE solver.}
Physics-informed machine learning can integrate seamlessly data and physical information (e.g. PDEs) to train neural network models.
Typical methods include physics-informed neural networks \cite{raissi2019pinn}, the Deep Ritz Method \cite{yu2018deep}, and the Deep Galerkin Method \cite{sirignano2018dgm} and their variants.
The integration of physics information is further combined in neural operators, such as the physics-informed DeepONet \cite{wang2021learning,goswami2023physics} and physics-informed FNO \cite{li2024physics,de2022generic}.
However, these methods integrate the residual of the differential equation into the loss function, which brings additional training difficulties, especially for solutions with multiscale phenomenon or discontinuity \cite{wang2022PINNfail,li2023implicit}.
Direct application of these methods may lead to poor performance for parametric interface problems.

\paragraph{Interface problem solver.}
For the piece-wise continuity property of the interface problems, piece-wise networks have been employed to approximate solutions by assigning a network to each subdomain and using jump conditions as penalty terms in the loss function \cite{guo2021deep,he2022mesh}. 
Despite their effectiveness, these methods can be complicated by the need to balance various loss terms during training. 
Domain decomposition methods have been adapted with neural networks to ease these challenges \cite{li2019d3m,li2020deep}. 
Additionally, approaches like the discontinuity capturing shallow neural network (DCSNN) \cite{hu2022discontinuity} and deep Nitsche-type method
\cite{wang2020mesh} have been proposed to manage high-contrast discontinuous coefficients and complex boundary conditions. 
However, these methods are confined to solving a particular interface equation.
Research in solving parametric interface problems is rather limited with some studies like IONet and physics-informed IONet \cite{wu2024solving} for piece-wise smooth solutions.
However, the need for large amount of labeled data and the training difficulties for physics-informed loss still exists.
Besides, if the interface problem exhibits singular perturbations, it will pose greater challenges to finding a solution.

\section{Preliminary}
\subsection{Interface Problems}
In this paper, we develop physics-informed machine learning methods to solve the following parametric interface problems:
\begin{equation}\label{eq3-1}
	\left\{\begin{aligned}
		&-\nabla\cdot(a\nabla u)+bu=f,\ \text{in }\Omega/ \Gamma,\\
		&[u]|_{\Gamma}=g_D,\ [a\nabla u\cdot\bm{n}]|_{\Gamma}=g_N,\ u|_{\partial\Omega}=0,
	\end{aligned}
	\right.
\end{equation}
where $\Omega$ is the domain comprising $N$ subdomains $\{\Omega_i, i=1,\dots,N\}$ separated by the interface $\Gamma$ as in Fig.~\ref{fig_interface_domain}. Here, $a(\bm{x}) > 0$ and $b(\bm{x}) \geq 0$ are piece-wise smooth functions. The terms $[\cdot]_{\Gamma}$ represents the jump across the interface, i.e., for a point $\bm{x}_0 \in \Gamma$ between $\Omega_i$ and $\Omega_{i+1}$,
\begin{equation}
    \begin{aligned}
        \left[u\right]|_{\bm{x}_0} &=u(\bm{x}_0^+)-u(\bm{x}_0^-) \\
&=\lim_{\bm{x}\in\Omega_{i+1},\bm{x}\to\bm{x}_0}u(\bm{x})-\lim_{\bm{x}\in\Omega_i,\bm{x}\to\bm{x}_0}u(\bm{x}).
    \end{aligned}\label{equ:jump definition}
\end{equation}

Our goal is to find the solution $u(x)$ quickly for different source function $f(x)$ (or the coefficient function $a(x)$, $b(x)$, jump/boundary conditions). Mathematically, we need to learn the mapping $\mathcal{G}:f(x) \rightarrow u(x)$.
Existing works shows it pose significant challenges due to the discontinuity of the solutions $u(x)$, particularly when they involve singular perturbations or high-contrast coefficients, resulting in intricate singularities that complicate resolution.

\subsection{Prior Physical information}
The interface equation \eqref{eq3-1} can be simplified using the coordinate transformation $\bm{y}(\bm{x}) = \int_{\cdot}^{\bm{x}} 1/a(\bm{\xi})d\bm{\xi}$, where the lower bound of the integral is determined by $\Omega$. The simplified form is:
\begin{equation}\label{eq3-2}
	\left\{\begin{aligned}
		&-\Delta u(\bm{y}) + c(\bm{y})u(\bm{y}) = F(\bm{y}),\ \text{in }\Omega/ \Gamma,\\
		&[u]|_{\Gamma}=g_D,\ [\nabla u\cdot\bm{n}]|_{\Gamma}=g_N,\ u|_{\partial\Omega}=0,
	\end{aligned}
	\right.
\end{equation}
where $c(\bm{y}) = a(\bm{x}(\bm{y}))b(\bm{x}(\bm{y}))$ represents a transformed coefficient and $F(\bm{y}) = a(\bm{x}(\bm{y}))f(\bm{x}(\bm{y}))$ is the transformed source term. 
We first show the prior physical information based on this reformulated equation in one dimension, then extend it to two dimension and above.

\begin{figure*}
    \centering
    \includegraphics[width=0.8\linewidth]{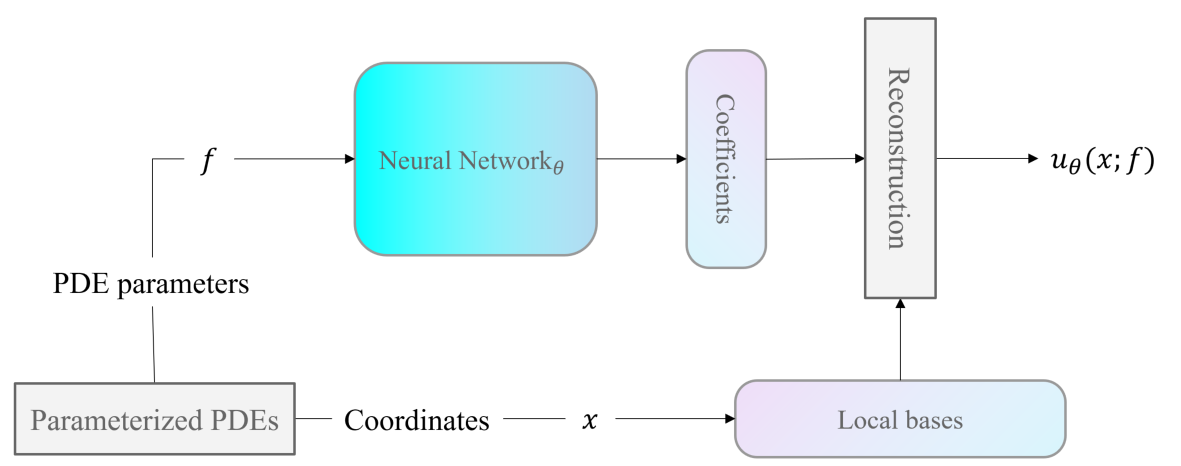}
    \caption{Architecture of PI-TFPNet.}
    \label{fig:architecture}
\end{figure*}

\paragraph{One dimension.} 
In a one-dimensional setting with a partition $\widetilde{\Gamma}=\{y_j\}$, the solution to Eq.~\eqref{eq3-2} within each sub-interval $(y_{j-1}, y_j)$ can be approximated by:
\begin{equation}\label{eq3-3}
u_h(y) = \alpha_j A_1^j(y) + \beta_j A_2^j(y) + \int_{y_{j-1}}^{y_j} F(s)G_j(y(x),s)ds,
\end{equation}
where $G_j$ is the Green's function associated with the sub-interval. The functions $A_1(y)=\{A_1^j(y)\}$ and $A_2(y)=\{A_2^j(y)\}$ serve as a pair of local basis functions, reflecting the distinct characteristics of the various sub-intervals. They are piece-wise defined for each sub-interval as follows:
\begin{equation}
    \begin{pmatrix}
        A_1^j(y) \\ A_2^j(y)
    \end{pmatrix}
     = \begin{cases}
        \begin{pmatrix}
        1, \\ y
        \end{pmatrix}, & p_j = q_j = 0, \\
        \begin{pmatrix}
        \exp(y\sqrt{q_j}) \\ \exp(-y\sqrt{q_j})
        \end{pmatrix}, & p_j = 0,\ q_j \ne 0, \\
        \begin{pmatrix}
        \text{Ai}(c_h(y)p_j^{-2/3}) \\ \text{Bi}(c_h(y)p_j^{-2/3})
        \end{pmatrix}, & p_j \ne 0,\ q_j \ne 0,
    \end{cases}
    \label{eq3-4}
\end{equation}
where $c_h(y)$ is a piece-wise linear approximation of the function $c(y)$, characterized by a slope $p_j$ and an intercept $q_j$ over the $j$-th sub-interval. 
Ai($\cdot$) and Bi($\cdot$) denote the Airy functions of the first and second kind, respectively. 
We note that $u_h(y)$ given by Eq.~$\eqref{eq3-3}$ can approximate the original solution $u(y)$ with high precision for proper coefficients $\{\alpha_j, \beta_j\}_j$.

\paragraph{Two dimension and above.} 
In two dimensions, we subdivide the domain $\Omega$ into a set of quadrilateral cells $\{\Delta_j\}$ and approximate the functions $c(x,y)$ and $F(x,y)$ as piece-wise constants. Specifically, within each cell $\Delta_j$, we approximate $c(x,y)$ and $F(x,y)$ by the constants $c^j_0$ and $F^j_0$, respectively. The solution to Eq.~\eqref{eq3-2} in this local region can then be approximated by:
\begin{equation}\label{eq3-5}
    u_h(x,y) = \frac{F^j_0}{c^j_0} + c_{j1}e^{\mu_j x} + c_{j2}e^{-\mu_j x} + c_{j3}e^{\mu_j y} +c_{j4}e^{-\mu_j y},
\end{equation}
where $\mu_j=\sqrt{c^j_0}$.
This expression can be straightforwardly extended to higher dimensions. Notably, $u_h(x,y)$ given by Eq.~$\eqref{eq3-5}$ can approximate the original solution $u(x,y)$ with high precision for proper coefficients $\{c_{j1}, c_{j2}, c_{j3}, c_{j4}\}_j$.

\section{Methods}
\paragraph{Definition of PI-TFPONet.}

Standard operator networks, such as DeepONet and FNO, learn the mapping $f\mapsto u$ using large datasets of paired $(f,u)$ values, equation information, or both. However, the jump conditions in the solution pose significant challenges for operator learning. In this paper, rather than learning the operator $\mathcal{G}: f \rightarrow u$ directly, we learn the mapping from $f$ to the discrete coefficients $\{\alpha_j,\beta_j\}_j$ in Eq.~\eqref{eq3-3} or $\{c_{j1},c_{j2},c_{j3},c_{j4}\}_j$ in Eq.~\eqref{eq3-5}. We then apply Eq.~\eqref{eq3-3} or Eq.~\eqref{eq3-5} to reconstruct the solution in each local region, denoting the reconstructed solution as $u_\theta$, as shown in the architecture in Fig.~\ref{fig:architecture}.

\begin{figure}[h]
	\centering
	\includegraphics[width=0.9\columnwidth]{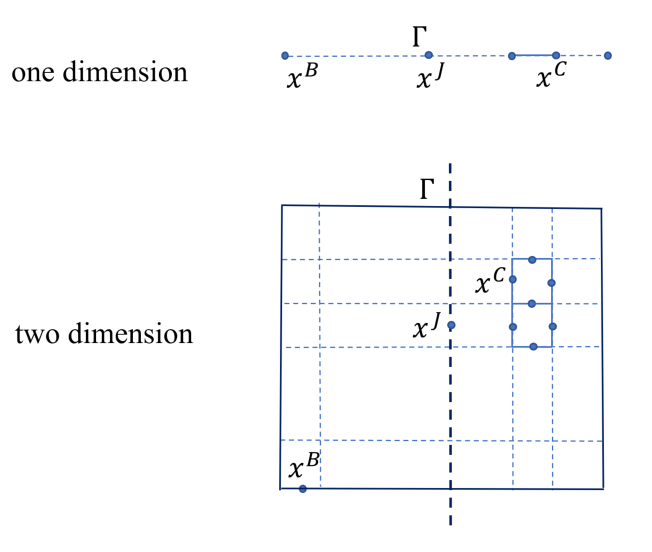}
	\caption{Illustration of the point sets $X^C$, $X^B$, and $X^J$.}
	\label{fig_local_mesh}
\end{figure}

\paragraph{Network Architecture.} 
We use a neural network to learn the mapping from $f(x)$ to the discrete coefficients. The input is the discrete vector $[f(\bm{x}_i)]_i$ and the output is the vectorized coefficients. For the one-dimensional case, we use a forward neural network, and for the two-dimensional case, we employ convolutional neural networks.

\paragraph{Loss Function.}

Since $u_\theta$ is defined piece-wise on $\{\Delta_i\}_i$ and adheres to a form like Eq.~\eqref{eq3-3} or Eq.~\eqref{eq3-5} within each cell, we must ensure $u_\theta$ is consistent at the shared edges (or points) of adjacent cells. We avoid using the equation residual in the loss function due to its known training challenges \cite{wang2022PINNfail}. We will later demonstrate that $u_\theta$ can approximate the solution accurately even without supervised data and equation residuals as constraints. Instead, we design an unsupervised loss function by enforcing continuity, boundary, and jump conditions. Specifically, the total loss function is:
\begin{equation}\label{eq3-6}
    \mathcal{L}(\theta) = \gamma_C\mathcal{L}_{C}(\theta) + \gamma_B\mathcal{L}_{B}(\theta) + \gamma_J\mathcal{L}_{J}(\theta),
\end{equation}
with $\gamma_C,\ \gamma_B,\ \gamma_J$ as penalizing parameters. The three components of the loss are:
\begin{equation}\label{eq3-7}
\left\{\begin{aligned}
    \mathcal{L}_{C}(\theta) =& \sum_{\bm{x} \in X^C}\left([u_{\theta}]|_{\bm{x}}\right)^2 + \left([\nabla u_{\theta}\cdot \bm{n}]|_{\bm{x}}\right)^2,\\
    \mathcal{L}_{B}(\theta) =& \sum_{\bm{x} \in X^B}\left(u_{\theta}(\bm{x}) - 0\right)^2,\\
    \mathcal{L}_{J}(\theta) =& \sum_{\bm{x} \in X^J}\left([u_{\theta}]|_{\bm{x}} - g_D(\bm{x})\right)^2 \\
    &+ \sum_{\bm{x} \in X^J}\left([\nabla u_{\theta}\cdot \bm{n}]|_{\bm{x}} - g_N(\bm{x})\right)^2.
	\end{aligned}
\right.
\end{equation}

Here, $X^C$ represents the set of points on the common edge of adjacent cells, excluding those on interfaces and boundaries. $X^J$ represents the set of interface points, while $X^B$ refers to the boundary points. Figure~\ref{fig_local_mesh} provides examples of $X^C$, $X^B$, and $X^J$ in one and two dimensions. The notation $[\cdot]|_{\bm{x}}$ denotes the jump across interfaces or common edges at $\bm{x}$, as defined in Eq.~\eqref{equ:jump definition}. This jump can be easily computed due to the piece-wise definition of $u_\theta$ on $\{\Delta_i\}_i$.

% with
% \begin{equation}\label{eq3-7}
% \left\{\begin{aligned}
%     &\mathcal{L}_{C}(\theta) = \sum_{i=1,y_i\neq x_0}^{N-1}\left|u_{\theta}(y_i^+)-u_{\theta}(y_i^-)\right|^2,\\
%     &\qquad\qquad + \sum_{i=1,y_i\neq x_0}^{N-1}\left|u'_{\theta}(y_i^+)-u'_{\theta}(y_i^-)\right|^2, \\
%     &\mathcal{L}_{B}(\theta) =\left|u_{\theta}(y_0)-0\right|^2 +\left|u_{\theta}(y_N)-0\right|^2,\\
%     &\mathcal{L}_{J}(\theta) =\left|u_{\theta}(x_0^+)-u_{\theta}(x_0^-)-g_D(x_0)\right|^2  \\
%     &\qquad\qquad+ \left|u_{\theta}'(x_0^+)-u_{\theta}'(x_0^-)-g_N(x_0)\right|^2.
% 	\end{aligned}
% 	\right.
% \end{equation}
% where $u_{\theta}(y)$ is given by Eq.~\eqref{eq3-3}
% \begin{align*}
%     u_{\theta}(y_j^-) \;=\; & \alpha_j(\theta) A_1^j(y_j) + \beta_j(\theta) A_2^j(y_j) \\ 
%     & + \int_{y_{j-1}}^{y_j} F(s)G_j(y(x),s)ds, \\
%     u_{\theta}(y_j^+) \;=\; & \alpha_{j+1}(\theta) A_1^{j+1}(y_j) + \beta_{j+1}(\theta) A_2^{j+1}(y_j) \\
%     & + \int_{y_{j}}^{y_{j+1}} F(s)G_{j+1}(y(x),s)ds,
% \end{align*}
% or given by Eq.~\eqref{eq3-5} similarly.

\paragraph{Data Assimilation}
Traditional numerical methods for interface problems cannot incorporate additional observed data $\{u(y_k)\}_{k=1}^{N_0}$. Our PI-TFPONet, being an unsupervised learning method, can seamlessly integrate with labeled data.
The supervised data loss $\mathcal{L}_D(\theta)=\frac{1}{N_0}\sum_{k=1}^{N_0}|u_{\theta}(y_k)-u(y_k)|^2$ can be naturally added to the unsupervised loss \eqref{eq3-6}.

\section{Theoretical Results}

By minimizing the unsupervised loss in Eq.~\eqref{eq3-6}, we determine the optimal neural network parameters for approximating the discrete coefficients. We then reconstruct the solution over the entire domain using Eq.~\eqref{eq3-3} or Eq.~\eqref{eq3-5}. Next, we present the error estimate for the reconstructed solution. We will consider the cases where the coefficient $a=1$ and $a=\varepsilon$, noting that other cases can be addressed through coordinate transformation.

\subsection{Convergence of PI-TFPONet}

The domain $\Omega = \bigcup\limits_{i=1}^M \Delta_i$ is partitioned into $M$ cells by $\widetilde{\Gamma}$. In the one-dimensional case, each $\Delta_i$ is an interval, and in the two-dimensional case, each $\Delta_i$ is a quadrilateral. We define the error function as $e=u_{\theta}-u$. Due to the piece-wise definition of $u_\theta$, $e$ may exhibit jumps between adjacent cells, not just at the interface $\Gamma$.

The norm is defined as:
\begin{equation*}
    \|u\|_{2, \Omega}^* = \left( \sum_{j=0}^l (\|u^{(l)}\|_{0,\Delta_1}^2+\cdots+\|u^{(l)}\|_{0,\Delta_M}^2)\right)^{1/2},
\end{equation*}
where $u^{(l)}$ denotes the $l$-th derivative of $u$, and $\|u\|_{0,\Delta_i}^2=\int_{\Delta_i}|u(\bm{x})|^2d\bm{x}$ is the standard $L^2$ norm. We then establish the following error estimate:

\begin{theorem}\label{thm:convergence}
The error estimate
\begin{equation}\label{eq4-1}
    \begin{aligned}
    \|e\|_{2, \Omega}^* &\le Ch^2\left(\|f\|_{0,\Omega}+\|g_D\|_{\infty,\Gamma}+\|g_N\|_{\infty,\Gamma}\right) \\
    &\quad + C\sqrt{\mathcal{L}_C(\theta)+\mathcal{L}_B(\theta)+\mathcal{L}_J(\theta)}
    \end{aligned}
\end{equation}
holds with a constant $C$ independent of $h$, $f$, $g_D$, and $g_N$.
\end{theorem}

\paragraph{Proof sketch.} We consider the one-dimensional problem as an example. Given that $u_{\theta}(x)$ from Eq.~\eqref{eq3-3} satisfies $-\Delta u_{\theta} + c_h\cdot u_{\theta} = f$ in each piece-wise domain, subtracting it from Eq.~\eqref{eq3-2} yields the following interface problem for the error function $e(x)$:
\begin{equation}\label{eq4-2}
\left\{\begin{aligned}
&-e''(x) + c_h(x)e(x) = R_h(x), &x\in \Omega/ \Gamma,\\
&[e]|_{x} = [u_{\theta}]|_{x},\ [e']|_{x} = [u_{\theta}']|_{x},&x\in \widetilde{\Gamma}/ (\Gamma\cup \partial\Omega),\\
&[e]|_{x} = [u_{\theta}]|_{x}-g_D(x), &x\in \Gamma,\\
&[e']|_{x} = [u_{\theta}']|_{x}-g_N(x),&x\in \Gamma,\\
&e(x) = u_{\theta}(x), &x\in\partial\Omega. 
\end{aligned}\right.
\end{equation}
Here, $R_h(x) = (c(x)-c_h(x))u(x)$, and by definition of $c_h$, we have $|c(x)-c_h(x)|\leq Ch^2$. Theorem 1 follows from the standard stability estimate for interface problems (detailed in the Appendix) and the loss in Eq.~\eqref{eq3-7}.

\subsection{Uniform Convergence of PI-TFPONet for Singular Perturbation Interface Problems}

For the singular perturbation interface problem with $0<\varepsilon\ll1$,
\begin{equation}\label{eq4-4}
	\left\{\begin{aligned}
		&-\varepsilon\Delta u(\bm{x}) + c(\bm{x})u(\bm{x}) = f(\bm{x}),\ \text{in }\Omega/ \Gamma,\\
		&[u]|_{\Gamma}=g_D,\quad [\varepsilon \nabla u\cdot \bm{n}]|_{\Gamma}=g_N,\quad u|_{\partial\Omega}=0.
	\end{aligned}
	\right.
\end{equation}
We define the norm
\begin{equation*}
    \|u\|_{\varepsilon, \Omega}^* = \left( \varepsilon(\|u\|_{1, \Omega}^*)^2 + (\|u\|^*_{0, \Omega})^2\right)^{1/2},
\end{equation*}
We then have the following error estimate, with the proof provided in the Appendix:
\begin{theorem}
The error estimate
\begin{equation}\label{eq4-5}
    \begin{aligned}
    \|e\|_{\varepsilon, \Omega}^* &\le Ch^2\left(\|f\|_{0,\Omega}+\|g_D\|_{\infty,\Gamma}+\|g_N\|_{\infty,\Gamma}\right) \\
    &\quad + C\sqrt{\mathcal{L}_C(\theta)+\mathcal{L}_B(\theta)+\mathcal{L}_J(\theta)}
    \end{aligned}
\end{equation}
holds with a constant $C$ independent of $\varepsilon,h,f,g_D$, and $g_N$.
\end{theorem}

\paragraph{Remark.} For small $\varepsilon$, boundary layers, interior layers, and corner layers with rapid transitions may arise, leading to poor approximation by vanilla operator networks. However, the above theoretical results guarantee that, provided the region $\Omega$ is finely divided and the loss function is minimized, our PI-TFPONet will produce a reconstructed solution that closely approximates the true solution at any point, not just at the grid points.

\section{Computational Results}
In this section we compares the results of our PI-TFPONet method to existing neural operators using various interface problems, including one-dimensional and two-dimensional cases, as well as scenarios with singular perturbation problems.

\paragraph{Error Metric.} To quantify the performance of our methods, we utilize the relative $L^2$ norm and relative $L^{\infty}$ norm as follows:
\begin{align}
\text{relative $L^2$ error} =  \sqrt{\frac{\sum_{i=1}^{N_r}|u_{\theta}(x_i)-u(x_i)|^2}{\sum_{i=1}^{N_r}|u(x_i)|^2}},\\
\text{relative $L^{\infty}$ error} =  \frac{\max_{i=1}^{N_r}|u_{\theta}(x_i)-u(x_i)|}{\max_{i=1}^{N_r}|u(x_i)|},
\end{align}
where $u(x)$ represents the ground truth derived from high-precision numerical methods \cite{huang2009tailored}.

\paragraph{Baselines.} Our PI-TFPONet method operates in an unsupervised manner. We compare its performance with several baseline methods, including supervised approaches such as the vanilla DeepONet \cite{lu2021learning} and interface operator networks (IONet) \cite{wu2024solving}. Additionally, we evaluate against unsupervised methods, specifically the physics-informed DeepONet \cite{wang2021learning} and the physics-informed IONet \cite{wu2024solving}.

\begin{table*}[!htb]
\centering
\scalebox{0.99}{
\begin{tabular}{ccccccc}
\toprule
Supervision & Method & 1D smooth  & 1D singular   & 1D high-contrast  & 2D interface & 2D singular\\ \midrule
supervised & DeepONet & 2.94e-01   & 2.85e-01   &  1.15e-01 & 6.41e-02  & 9.14e-02\\ 
supervised & IONet &  \textbf{1.77e-03}  &  4.71e-02  & 2.00e-02   & \textbf{5.50e-03} & 4.97e-02 \\ 
unsupervised & PI-DeepONet & 9.64e-01   &  1.01e-00  & 5.87e-01 & 7.51e-01 & 2.19e-01 \\ 
unsupervised & PI-IONet &  4.49e-03  & 7.99e-01   &  4.88e-01 &  3.89e-02 &1.46e-01 \\ \midrule
unsupervised & PI-TFPONet & 2.93e-03   & \textbf{9.61e-03}   & \textbf{4.88e-03}  & 7.64e-03&\textbf{1.07e-02}  \\ \bottomrule
\end{tabular}}
\caption{A comparison of the relative $L^{2}$ errors for different benchmarks and models.}
\label{tab_results}
\end{table*}

\paragraph{Neural Networks.} In the one-dimensional case, we employ a fully connected neural network (FNN) with 4 hidden layers, each containing 64 neurons, and ReLU activation functions. For the two-dimensional case, the input function $f$ is discretized into a matrix and processed using a convolutional neural network (CNN). This CNN features an encoder-decoder architecture with a latent vector of 256 dimensions, and both the encoder and decoder consist of 4 convolutional layers. Detailed network architecture is provided in the appendix.

\paragraph{Training Details.} We utilize the AdamW optimizer \cite{loshchilovdecoupled} with momentum parameters $\beta_1=0.9$ and $\beta_2=0.999$, and a weight decay of $1e-4$ to update all network parameters. The initial learning rate is set to $1e-4$ with exponential decay. The Gaussian random field (GRF) generates the source function $f(x)\sim\mathcal{G}(0,k_l(x_1,x_2))$, where the covariance kernel $k_l(x_1,x_2)=\exp(-(x_1-x_2)^2/(2l^2))$. PI-TFPONets are trained on 1000 samples of $f$ with meshes of $32\times1$ for one dimension and $16\times 16$ for two dimensions. The trained models are tested on 200 different samples of $f$ using finer meshes of $256\times1$ for one dimension and $128\times128$ for two dimensions.

\subsection{One-Dimensional Settings}
For the one-dimensional case, we will conduct experiments on the following problem:
\begin{equation}\label{eq1-1d}
	\left\{\begin{aligned}
		&-a(x) u''(x) + b(x)u(x) = f(x),\ \text{in }\Omega/ \Gamma,\\
		&[u]|_{\Gamma}=1,\quad [u']|_{\Gamma}=1, \\
  &u(0)=u(1)=0.
	\end{aligned}
	\right.
\end{equation}
Here, $\Omega=[0,1]$ and $\Gamma=\{0.5\}$. In different experiments, the coefficients $a(x)$ and $b(x)$ will take different values, causing the solution of the corresponding equation to exhibit different characteristics.

\paragraph{1D smooth.}
We first consider the simplest type of interface problem, where the solution is smooth within each sub-region and does not exhibit drastic changes. Specifically, the coefficient $a(x)=1$, and $b(x)$ is defined by the following formula:
\begin{equation*}
    b(x) = 
    \begin{cases}
        1+e^x, & 0<x<0.5, \\
        1-\ln(x+1), & 0.5< x<1.
    \end{cases}
\end{equation*}
\begin{figure}[htbp]
	\centering
	\subfloat[]{\includegraphics[width=.45\columnwidth]{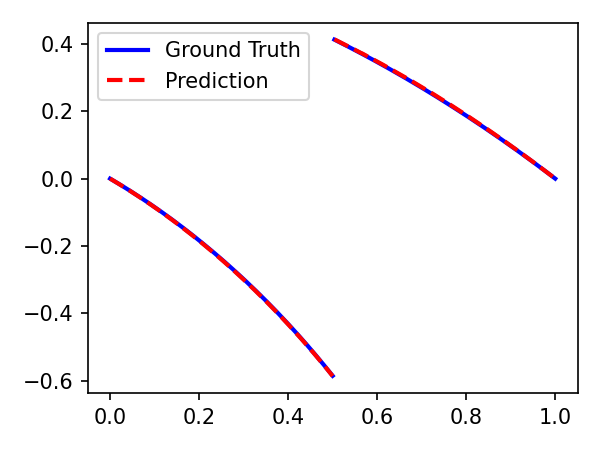}}
	\subfloat[]{\includegraphics[width=.45\columnwidth]{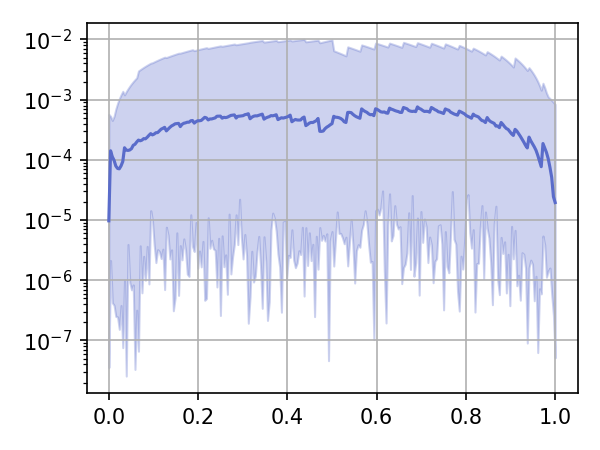}}
	\caption{[1D smooth] (a)PI-TFPONet’s refinement prediction. (b)Error distribution across 200 test examples. Solid line: median error, shaded area: min to max error range.}
 \label{fig:1d-smooth}
\end{figure}
We use 1000 $f$ samples to train PI-TFPONet on a coarse grid with a resolution of 32 and then test it using $f$ samples independent of the training set. For the test sample $f$, we use the model's output to reconstruct the solution and evaluate it on a fine grid with a resolution of 256. Fig.~\ref{fig:1d-smooth}(a) shows the predicted solution and the ground truth for a test sample, while Fig.~\ref{fig:1d-smooth}(b) displays the absolute error between the predicted solution and the ground truth on the fine grid for all 200 test samples. The point-wise errors are all smaller than $1e-02$, primarily ranging between $1e-05$ and $1e-02$, indicating that the model trained on the coarse grid can reconstruct a highly accurate solution.
\paragraph{1D singular.}Next, we consider a more complex scenario. If the coefficient $a(x)$ satisfies $0<a(x)\ll1$, the solution may develop boundary layers or inner layers at $x=0,0.5,1$. This means the solution changes rapidly in one or more very narrow regions, posing challenges for solving the problem. Here, we set $a(x)=0.001$ and define the coefficient $b(x)$ using the following piece-wise function:
\begin{equation*}
    b(x) = 
    \begin{cases}
        5, & 0<x<0.5, \\
        0.1\cdot(4+32x), & 0.5< x<1.
    \end{cases}
\end{equation*}
In scenarios with boundary layers or inner layers, traditional numerical methods, as well as vanilla PINN and DeepONet, require an increased number of sampling points on $[0,1]$ to capture the fine layer details. In contrast, we train our PI-DeepONet on a coarse grid with a resolution of 32 and test it on a fine grid with a resolution of 256. Fig.~\ref{fig:1d-singular}(a) presents a set of prediction and ground truth, showing a high degree of similarity. Fig.~\ref{fig:1d-singular}(b) illustrates the point-wise absolute error between the predictions and ground truth for 200 samples. The error values are all below $1e-02$, with most concentrated between $1e-06$ and $1e-03$. This demonstrates that our approach can effectively capture thin layer information at a relatively low computational cost.
\begin{figure}[htbp]
	\centering
	\subfloat[]{\includegraphics[width=.45\columnwidth]{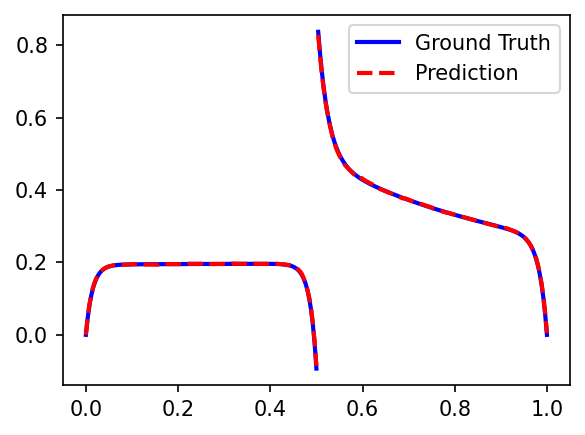}}
	\subfloat[]{\includegraphics[width=.45\columnwidth]{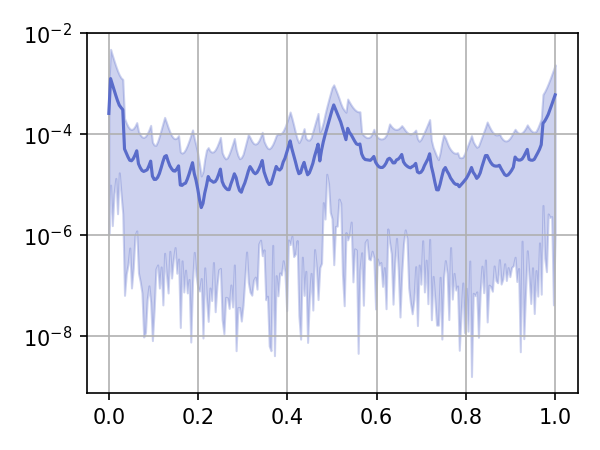}}
	\caption{[1D singular] (a)PI-TFPONet’s refinement prediction. (b)Error distribution across 200 test examples. Solid line: median error, shaded area: min to max error range.} 
 \label{fig:1d-singular}
\end{figure}
\paragraph{1D high-contrast.}
Next, we consider a problem where the coefficient $a(x)$ satisfies $0<a(x)\ll1$ in one sub-region and is moderate in another, resulting in a high contrast between the regions. This is known as a high contrast problem. Such problems exhibit boundary or inner layers, and the change in $a(x)$ introduces additional singularities at the interface. Here, we define the coefficients $a(x)$ and $b(x)$ as the following piece-wise functions:
\begin{equation*}
    a(x) = 
    \begin{cases}
        0.001, & 0<x<0.5, \\
        1, & 0.5< x<1,
    \end{cases}
\end{equation*}
\begin{equation*}
    b(x) = 
    \begin{cases}
        2x+1, & 0<x<0.5 \\
        2(1-x)+1, & 0.5< x<1.
    \end{cases}
\end{equation*}
Fig.~\ref{fig:1d high-contrast} illustrates the model's predictive performance. Subfigure (a) shows a set of reconstructed and reference solutions, which are very close. Subfigure (b) plots the point-wise absolute errors for 200 reconstructed and reference solutions on a grid with a resolution of 257. The errors are all less than $1e-01$, with most ranging between $1e-05$ and $1e-02$, indicating that our model achieves high accuracy at lower training resolutions.
\begin{figure}[htbp]
	\centering
	\subfloat[]{\includegraphics[width=.45\columnwidth]{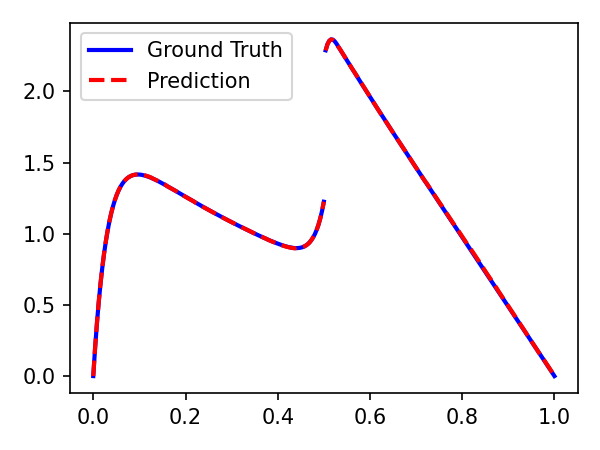}}
	\subfloat[]{\includegraphics[width=.45\columnwidth]{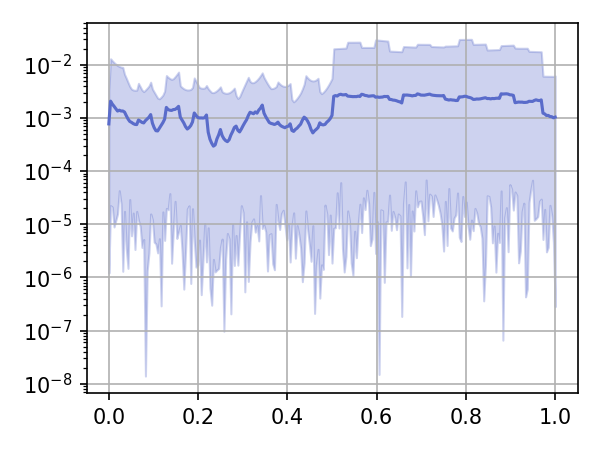}}
	\caption{[1D high-contrast] (a)PI-TFPONet’s refinement prediction. (b)Error distribution across 200 test examples. Solid line: median error, shaded area: min to max error range.} 
 \label{fig:1d high-contrast}
\end{figure}

\subsection{Two-Dimensional Settings}
For the two-dimensional case, we will perform experiments on the following problem:
\begin{equation}\label{eq6-2d}
	\left\{\begin{aligned}
		&-a(x,y)\cdot\Delta u(x,y) + b(x,y)u(x,y) = f(x,y),\ \text{in }\Omega/ \Gamma,\\
		&[u]|_{\Gamma}=1,\quad [\nabla u\cdot \bm{n}]|_{\Gamma}=0, \\
  &u|_{\partial\Omega}=g_B.
	\end{aligned}
	\right.
\end{equation}
Here, $\Omega=[0,1]\times[0,1]$, $\Gamma=\{x=0.5,0<y<1\}$ and
\begin{equation*}
    b(x,y) = 
    \begin{cases}
        16, & 0< x< 0.5,\ 0<y<1, \\
        1, & 0.5< x<1,\ 0<y<1,
    \end{cases}
\end{equation*}
and
\begin{equation*}
    g_B(x,y) = 
    \begin{cases}
        2(1-x), & 0.5< x<1,\ y=\{0,1\}, \\
        0, & \text{othwise}.
    \end{cases}
\end{equation*}
Subsequently, we will evaluate the model's performance on general interface problems as well as those with singular perturbations by varying the coefficient $a(x,y)$.

\paragraph{2D interface.}
We begin by considering the coefficient $a(x,y)=1$ as a constant, resulting in a piece-wise smooth solution over two non-intersecting sub-regions. We train PI-TFPONet with 1000 
$f$ samples on a coarse $16\times16$ grid and test it on 200 $f$ samples independent of the training set. Using the model's output, we reconstruct the solution according to formula (5) and test it on a fine $128\times128$ grid. Fig.~\ref{fig:2d_interface} presents a set of prediction and ground truth, along with the point-wise absolute errors between them. Despite the coarse training resolution, the predicted solution closely approximates the reference solution at a higher resolution, demonstrating high accuracy.

\begin{figure}[htbp]
	\centering
	\begin{minipage}{0.45\columnwidth}
		\subfloat[]{\includegraphics[width=\textwidth]{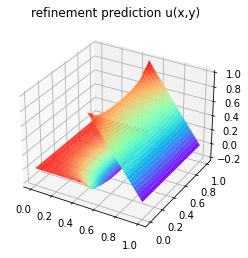}}
	\end{minipage}\quad	
	\begin{minipage}{0.45\columnwidth}
		\subfloat[]{\includegraphics[width=\textwidth]{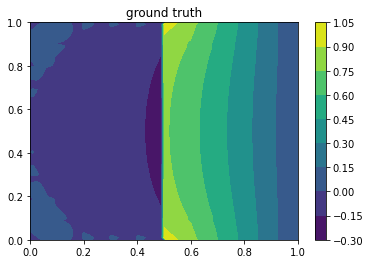}}\\
		\subfloat[]{\includegraphics[width=\textwidth]{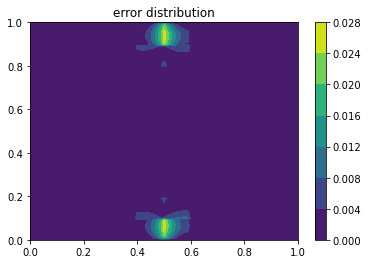}}
	\end{minipage}
	\caption{[2D interface]: (a)PI-TFPONet's refinement predicted solution. (b)ground truth solution. (c)error distribution over domain.}
 \label{fig:2d_interface}
\end{figure}

\paragraph{2D singular.}
Next, we consider $a(x,y)=0.001$, where the solution may exhibit boundary or inner layers. We train PI-TFPONet with 1000 samples of $f$ on a coarse $16\times16$ grid. Figure \ref{fig:2d_singular} presents the predicted and reference solutions, along with the point-wise absolute error between them. The error does not exceed $3.5e-02$, even in the boundary layer region. Despite being trained at a lower resolution, these results demonstrate that our model possesses sufficient prior information to generalize well at a resolution eight times higher.

\begin{figure}[htbp]
	\centering
	\begin{minipage}{0.45\columnwidth}
		\subfloat[]{\includegraphics[width=\textwidth]{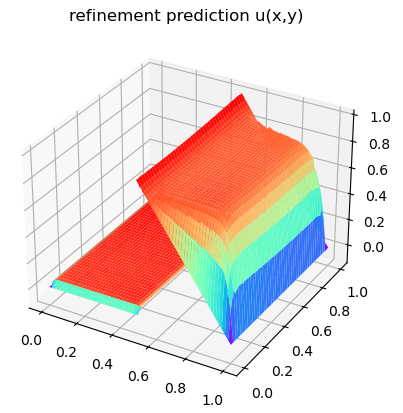}}
	\end{minipage}\quad	
	\begin{minipage}{0.45\columnwidth}
		\subfloat[]{\includegraphics[width=\textwidth]{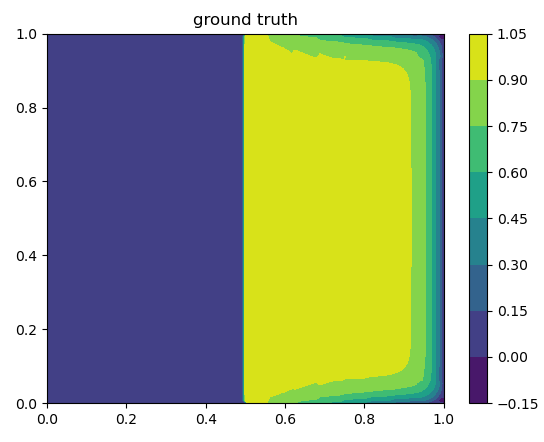}}\\
		\subfloat[]{\includegraphics[width=\textwidth]{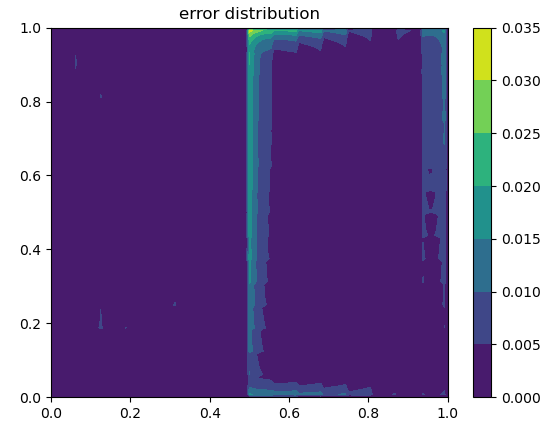}}
	\end{minipage}
	\caption{[2D singular]: (a)PI-TFPONet's refinement predicted solution. (b)ground truth solution. (c)error distribution over domain.}
 \label{fig:2d_singular}
\end{figure}

\subsection{Results}
Table~\ref{tab_results} presents the relative $L^2$ error of our model and the baseline models on the test set for all the aforementioned problems. Additional relative $L^{\infty}$ error data can be found in the appendix. The bold numbers in each column indicate the smallest error among all models tested.

For interface problems with additional singularities, including the 1D singular, 1D high-contrast, and 2D singular cases, our model achieves the highest accuracy, surpassing even those trained with supervised data. Models like DeepONet and IONet struggle to capture boundary or inner layer information, likely due to insufficient resolution in the training samples. In contrast, our model, despite being trained at low resolution, maintains high accuracy at a test resolution eight times higher.

For general interface problems, our model is slightly less accurate than IONet but more accurate than other models. However, given that our model does not require additional supervised data and uses only one network, this level of accuracy is acceptable.

\section{Conclusion}

This study introduces an unsupervised PI-DeepONet for solving parameterized interface problems. The solution is reconstructed on a coarse grid by learning the coefficients of the local bases. Our theoretical analysis shows that the reconstructed solution not only satisfies the equation exactly at the predefined grid points but also approximates the exact solution with high accuracy in other regions. Numerical experiments on various problems demonstrate that our model generalizes better than current state-of-the-art models for interface problems with additional singularities, such as singular perturbations or high contrast issues. For general interface problems, our approach performs slightly below the best existing models. However, unlike these state-of-the-art models, our model operates entirely unsupervised, eliminating the need for large amounts of supervised data.

In our experiments, we chose $\Omega=\Omega_1\bigcup\Omega_2$ for illustration. While the number of sub-regions can vary, this does not affect our model's validity, so we use two sub-regions as an example. We input the source term $f$, fix the local bases, and output the coefficients of these bases. If functions $a$ or $b$ are also input along with $f$, the local bases become variable. In such cases, integrating a network inspired by meta-learning to adaptively learn the local bases is a promising direction for future research. This approach could address more complex problems where determining the local bases is not straightforward. 

\bibliography{arxiv}

\begin{thebibliography}{39}
\providecommand{\natexlab}[1]{#1}

\bibitem[{Azizzadenesheli et~al.(2024)Azizzadenesheli, Kovachki, Li, Liu-Schiaffini, Kossaifi, and Anandkumar}]{azizzadenesheli2024neural}
Azizzadenesheli, K.; Kovachki, N.; Li, Z.; Liu-Schiaffini, M.; Kossaifi, J.; and Anandkumar, A. 2024.
\newblock Neural operators for accelerating scientific simulations and design.
\newblock \emph{Nature Reviews Physics}, 1--9.

\bibitem[{Babu{\v{s}}ka(1970)}]{babuvska1970finite}
Babu{\v{s}}ka, I. 1970.
\newblock The finite element method for elliptic equations with discontinuous coefficients.
\newblock \emph{Computing}, 5(3): 207--213.

\bibitem[{Chen and Chen(1995)}]{chen1995universal}
Chen, T.; and Chen, H. 1995.
\newblock Universal approximation to nonlinear operators by neural networks with arbitrary activation functions and its application to dynamical systems.
\newblock \emph{IEEE transactions on neural networks}, 6(4): 911--917.

\bibitem[{De~Ryck and Mishra(2022)}]{de2022generic}
De~Ryck, T.; and Mishra, S. 2022.
\newblock Generic bounds on the approximation error for physics-informed (and) operator learning.
\newblock \emph{Advances in Neural Information Processing Systems}, 35: 10945--10958.

\bibitem[{Deng et~al.(2022)Deng, Shin, Lu, Zhang, and Karniadakis}]{deng2022approximation}
Deng, B.; Shin, Y.; Lu, L.; Zhang, Z.; and Karniadakis, G.~E. 2022.
\newblock Approximation rates of DeepONets for learning operators arising from advection--diffusion equations.
\newblock \emph{Neural Networks}, 153: 411--426.

\bibitem[{Evans(2022)}]{evans2022partial}
Evans, L.~C. 2022.
\newblock \emph{Partial differential equations}, volume~19.
\newblock American Mathematical Society.

\bibitem[{Fadlun et~al.(2000)Fadlun, Verzicco, Orlandi, and Mohd-Yusof}]{fadlun2000combined}
Fadlun, E.~A.; Verzicco, R.; Orlandi, P.; and Mohd-Yusof, J. 2000.
\newblock Combined immersed-boundary finite-difference methods for three-dimensional complex flow simulations.
\newblock \emph{Journal of Computational Physics}, 161(1): 35--60.

\bibitem[{Goswami et~al.(2023)Goswami, Bora, Yu, and Karniadakis}]{goswami2023physics}
Goswami, S.; Bora, A.; Yu, Y.; and Karniadakis, G.~E. 2023.
\newblock Physics-informed deep neural operator networks.
\newblock In \emph{Machine Learning in Modeling and Simulation: Methods and Applications}, 219--254. Springer.

\bibitem[{Guo and Yang(2021)}]{guo2021deep}
Guo, H.; and Yang, X. 2021.
\newblock Deep unfitted Nitsche method for elliptic interface problems.
\newblock \emph{arXiv preprint arXiv:2107.05325}.

\bibitem[{Haghighat, bin Waheed, and Karniadakis(2024)}]{haghighat2024deeponet}
Haghighat, E.; bin Waheed, U.; and Karniadakis, G. 2024.
\newblock En-DeepONet: An enrichment approach for enhancing the expressivity of neural operators with applications to seismology.
\newblock \emph{Computer Methods in Applied Mechanics and Engineering}, 420: 116681.

\bibitem[{He, Hu, and Mu(2022)}]{he2022mesh}
He, C.; Hu, X.; and Mu, L. 2022.
\newblock A mesh-free method using piecewise deep neural network for elliptic interface problems.
\newblock \emph{Journal of Computational and Applied Mathematics}, 412: 114358.

\bibitem[{Hesthaven(2003)}]{hesthaven2003high}
Hesthaven, J.~S. 2003.
\newblock High-order accurate methods in time-domain computational electromagnetics: A review.
\newblock \emph{Advances in Imaging and Electron Physics}, 127: 59--123.

\bibitem[{Hu, Lin, and Lai(2022)}]{hu2022discontinuity}
Hu, W.-F.; Lin, T.-S.; and Lai, M.-C. 2022.
\newblock A discontinuity capturing shallow neural network for elliptic interface problems.
\newblock \emph{Journal of Computational Physics}, 469: 111576.

\bibitem[{Huang(2009)}]{huang2009tailored}
Huang, Z. 2009.
\newblock Tailored finite point method for the interface problem.
\newblock \emph{Networks and Heterogeneous Media}, 4(1): 91--106.

\bibitem[{Ji et~al.(2018)Ji, Liu, Xu, Shen, and Lu}]{ji2018finite}
Ji, N.; Liu, T.; Xu, J.; Shen, L.~Q.; and Lu, B. 2018.
\newblock A finite element solution of lateral periodic Poisson--Boltzmann model for membrane channel proteins.
\newblock \emph{International Journal of Molecular Sciences}, 19(3): 695.

\bibitem[{Kellogg(1971)}]{kellogg1971singularities}
Kellogg, R. 1971.
\newblock Singularities in interface problems.
\newblock In \emph{Numerical Solution of Partial Differential Equations--II}, 351--400. Elsevier.

\bibitem[{Kovachki, Lanthaler, and Mishra(2021)}]{kovachki2021universal}
Kovachki, N.; Lanthaler, S.; and Mishra, S. 2021.
\newblock On universal approximation and error bounds for Fourier neural operators.
\newblock \emph{Journal of Machine Learning Research}, 22(290): 1--76.

\bibitem[{Lanthaler, Mishra, and Karniadakis(2022)}]{lanthaler2022error}
Lanthaler, S.; Mishra, S.; and Karniadakis, G.~E. 2022.
\newblock Error estimates for deeponets: A deep learning framework in infinite dimensions.
\newblock \emph{Transactions of Mathematics and Its Applications}, 6(1): tnac001.

\bibitem[{Li et~al.(2019)Li, Tang, Wu, and Liao}]{li2019d3m}
Li, K.; Tang, K.; Wu, T.; and Liao, Q. 2019.
\newblock D3M: A deep domain decomposition method for partial differential equations.
\newblock \emph{IEEE Access}, 8: 5283--5294.

\bibitem[{Li, Xiang, and Xu(2020)}]{li2020deep}
Li, W.; Xiang, X.; and Xu, Y. 2020.
\newblock Deep domain decomposition method: Elliptic problems.
\newblock In \emph{Mathematical and Scientific Machine Learning}, 269--286. PMLR.

\bibitem[{Li, Chen, and Huang(2023)}]{li2023implicit}
Li, Y.; Chen, S.-C.; and Huang, S.-J. 2023.
\newblock Implicit stochastic gradient descent for training physics-informed neural networks.
\newblock In \emph{Proceedings of the AAAI Conference on Artificial Intelligence}, volume~37, 8692--8700.

\bibitem[{Li et~al.(2024{\natexlab{a}})Li, Kovachki, Choy, Li, Kossaifi, Otta, Nabian, Stadler, Hundt, Azizzadenesheli et~al.}]{li2024geometry}
Li, Z.; Kovachki, N.; Choy, C.; Li, B.; Kossaifi, J.; Otta, S.; Nabian, M.~A.; Stadler, M.; Hundt, C.; Azizzadenesheli, K.; et~al. 2024{\natexlab{a}}.
\newblock Geometry-informed neural operator for large-scale 3d pdes.
\newblock \emph{Advances in Neural Information Processing Systems}, 36.

\bibitem[{Li et~al.(2020)Li, Kovachki, Azizzadenesheli, Bhattacharya, Stuart, Anandkumar et~al.}]{li2020fourier}
Li, Z.; Kovachki, N.~B.; Azizzadenesheli, K.; Bhattacharya, K.; Stuart, A.; Anandkumar, A.; et~al. 2020.
\newblock Fourier neural operator for parametric partial differential equations.
\newblock In \emph{International Conference on Learning Representations}.

\bibitem[{Li et~al.(2024{\natexlab{b}})Li, Zheng, Kovachki, Jin, Chen, Liu, Azizzadenesheli, and Anandkumar}]{li2024physics}
Li, Z.; Zheng, H.; Kovachki, N.; Jin, D.; Chen, H.; Liu, B.; Azizzadenesheli, K.; and Anandkumar, A. 2024{\natexlab{b}}.
\newblock Physics-informed neural operator for learning partial differential equations.
\newblock \emph{ACM/JMS Journal of Data Science}, 1(3): 1--27.

\bibitem[{Liu et~al.(2020)Liu, Sussman, Lian, and {Yousuff Hussaini}}]{LIU2020109017}
Liu, Y.; Sussman, M.; Lian, Y.; and {Yousuff Hussaini}, M. 2020.
\newblock A moment-of-fluid method for diffusion equations on irregular domains in multi-material systems.
\newblock \emph{Journal of Computational Physics}, 402: 109017.

\bibitem[{Loshchilov and Hutter(2018)}]{loshchilovdecoupled}
Loshchilov, I.; and Hutter, F. 2018.
\newblock Decoupled Weight Decay Regularization.
\newblock In \emph{International Conference on Learning Representations}.

\bibitem[{Lu et~al.(2021)Lu, Jin, Pang, Zhang, and Karniadakis}]{lu2021learning}
Lu, L.; Jin, P.; Pang, G.; Zhang, Z.; and Karniadakis, G.~E. 2021.
\newblock Learning nonlinear operators via DeepONet based on the universal approximation theorem of operators.
\newblock \emph{Nature machine intelligence}, 3(3): 218--229.

\bibitem[{Lu et~al.(2022)Lu, Meng, Cai, Mao, Goswami, Zhang, and Karniadakis}]{lu2022comprehensive}
Lu, L.; Meng, X.; Cai, S.; Mao, Z.; Goswami, S.; Zhang, Z.; and Karniadakis, G.~E. 2022.
\newblock A comprehensive and fair comparison of two neural operators (with practical extensions) based on fair data.
\newblock \emph{Computer Methods in Applied Mechanics and Engineering}, 393: 114778.

\bibitem[{Raissi, Perdikaris, and Karniadakis(2019)}]{raissi2019pinn}
Raissi, M.; Perdikaris, P.; and Karniadakis, G.~E. 2019.
\newblock Physics-informed neural networks: A deep learning framework for solving forward and inverse problems involving nonlinear partial differential equations.
\newblock \emph{Journal of Computational physics}, 378: 686--707.

\bibitem[{Shukla et~al.(2024)Shukla, Oommen, Peyvan, Penwarden, Plewacki, Bravo, Ghoshal, Kirby, and Karniadakis}]{shukla2024deep}
Shukla, K.; Oommen, V.; Peyvan, A.; Penwarden, M.; Plewacki, N.; Bravo, L.; Ghoshal, A.; Kirby, R.~M.; and Karniadakis, G.~E. 2024.
\newblock Deep neural operators as accurate surrogates for shape optimization.
\newblock \emph{Engineering Applications of Artificial Intelligence}, 129: 107615.

\bibitem[{Sirignano and Spiliopoulos(2018)}]{sirignano2018dgm}
Sirignano, J.; and Spiliopoulos, K. 2018.
\newblock DGM: A deep learning algorithm for solving partial differential equations.
\newblock \emph{Journal of Computational Physics}, 375: 1339--1364.

\bibitem[{Sussman and Fatemi(1999)}]{sussman1999efficient}
Sussman, M.; and Fatemi, E. 1999.
\newblock An efficient, interface-preserving level set redistancing algorithm and its application to interfacial incompressible fluid flow.
\newblock \emph{SIAM Journal on Scientific Computing}, 20(4): 1165--1191.

\bibitem[{Wang et~al.(2019)Wang, Zheng, Lu, and Shi}]{WANG2019117}
Wang, L.; Zheng, H.; Lu, X.; and Shi, L. 2019.
\newblock A Petrov-Galerkin finite element interface method for interface problems with Bloch-periodic boundary conditions and its application in phononic crystals.
\newblock \emph{Journal of Computational Physics}, 393: 117--138.

\bibitem[{Wang, Wang, and Perdikaris(2021)}]{wang2021learning}
Wang, S.; Wang, H.; and Perdikaris, P. 2021.
\newblock Learning the solution operator of parametric partial differential equations with physics-informed DeepONets.
\newblock \emph{Science advances}, 7(40): eabi8605.

\bibitem[{Wang, Wang, and Perdikaris(2022)}]{wang2022improved}
Wang, S.; Wang, H.; and Perdikaris, P. 2022.
\newblock Improved architectures and training algorithms for deep operator networks.
\newblock \emph{Journal of Scientific Computing}, 92(2): 35.

\bibitem[{Wang, Yu, and Perdikaris(2022)}]{wang2022PINNfail}
Wang, S.; Yu, X.; and Perdikaris, P. 2022.
\newblock When and why PINNs fail to train: A neural tangent kernel perspective.
\newblock \emph{Journal of Computational Physics}, 449: 110768.

\bibitem[{Wang and Zhang(2020)}]{wang2020mesh}
Wang, Z.; and Zhang, Z. 2020.
\newblock A mesh-free method for interface problems using the deep learning approach.
\newblock \emph{Journal of Computational Physics}, 400: 108963.

\bibitem[{Wu et~al.(2024)Wu, Zhu, Tang, and Lu}]{wu2024solving}
Wu, S.; Zhu, A.; Tang, Y.; and Lu, B. 2024.
\newblock Solving parametric elliptic interface problems via interfaced operator network.
\newblock \emph{Journal of Computational Physics}, 113217.

\bibitem[{Yu et~al.(2018)}]{yu2018deep}
Yu, B.; et~al. 2018.
\newblock The deep Ritz method: a deep learning-based numerical algorithm for solving variational problems.
\newblock \emph{Communications in Mathematics and Statistics}, 6(1): 1--12.

\end{thebibliography}

\newpage
\appendix
\onecolumn

% \documentclass{article}
% \usepackage[english]{babel}

% % Set page size and margins
% % Replace `letterpaper' with `a4paper' for UK/EU standard size
% \usepackage[a4paper,top=2cm,bottom=2cm,left=3cm,right=3cm,marginparwidth=1.75cm]{geometry}
% \usepackage{amsmath}
% \usepackage{bm}
% \usepackage{amssymb}
% \usepackage{mathtools}
% \usepackage{amsthm}
% \newtheorem{theorem}{Theorem}
% \newtheorem{lemma}{Lemma}
% \usepackage{enumerate}
% \usepackage{graphicx}
% \usepackage{indentfirst}
% \usepackage{enumerate}
% \usepackage{multirow}
% \usepackage{booktabs}
% \usepackage{enumitem}
% \usepackage{subfig}
% \usepackage{float}
% \usepackage[colorlinks=true, allcolors=blue]{hyperref}
% \usepackage{cite}

% \title{Appendix}
% \date{}
% \begin{document}
% \setcounter{equation}{16}
% \maketitle

\section{Additional Methods}

Figures \ref{fig_Net1d} and \ref{fig_Net2d} illustrate the application of PI-TFPONet to one-dimensional and two-dimensional problems, respectively, and outline the problem framework. Table \ref{table:componets} provides a detailed structure of the convolutional neural network.

\begin{figure*}[!tbh]
	\centering
	\includegraphics[width=0.7\textwidth]{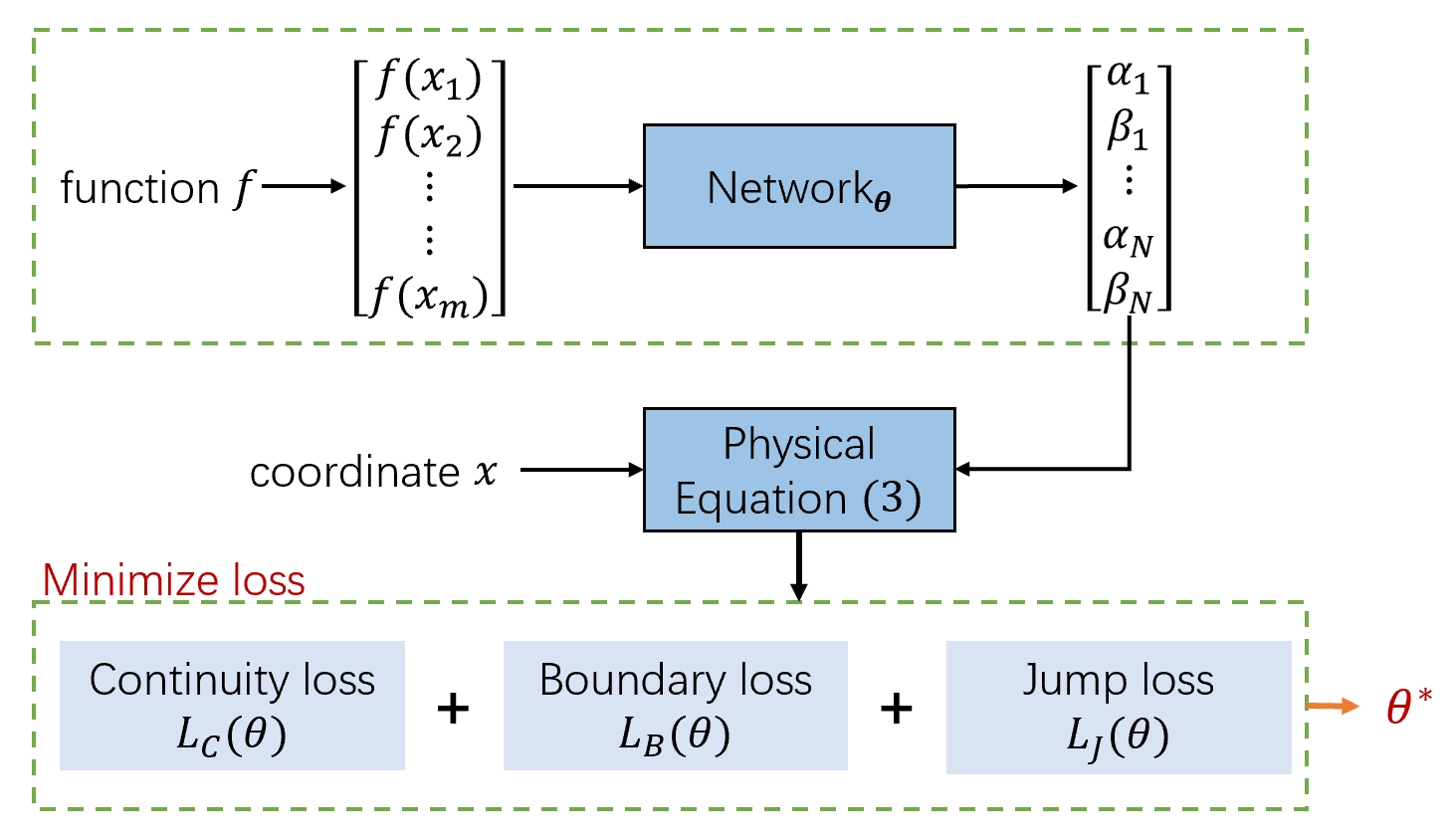}
	\caption{Architecture of PI-TFPNet for one dimensional case.}
	\label{fig_Net1d}
\end{figure*}

\begin{figure*}[!tbh]
	\centering
	\includegraphics[width=0.7\textwidth]{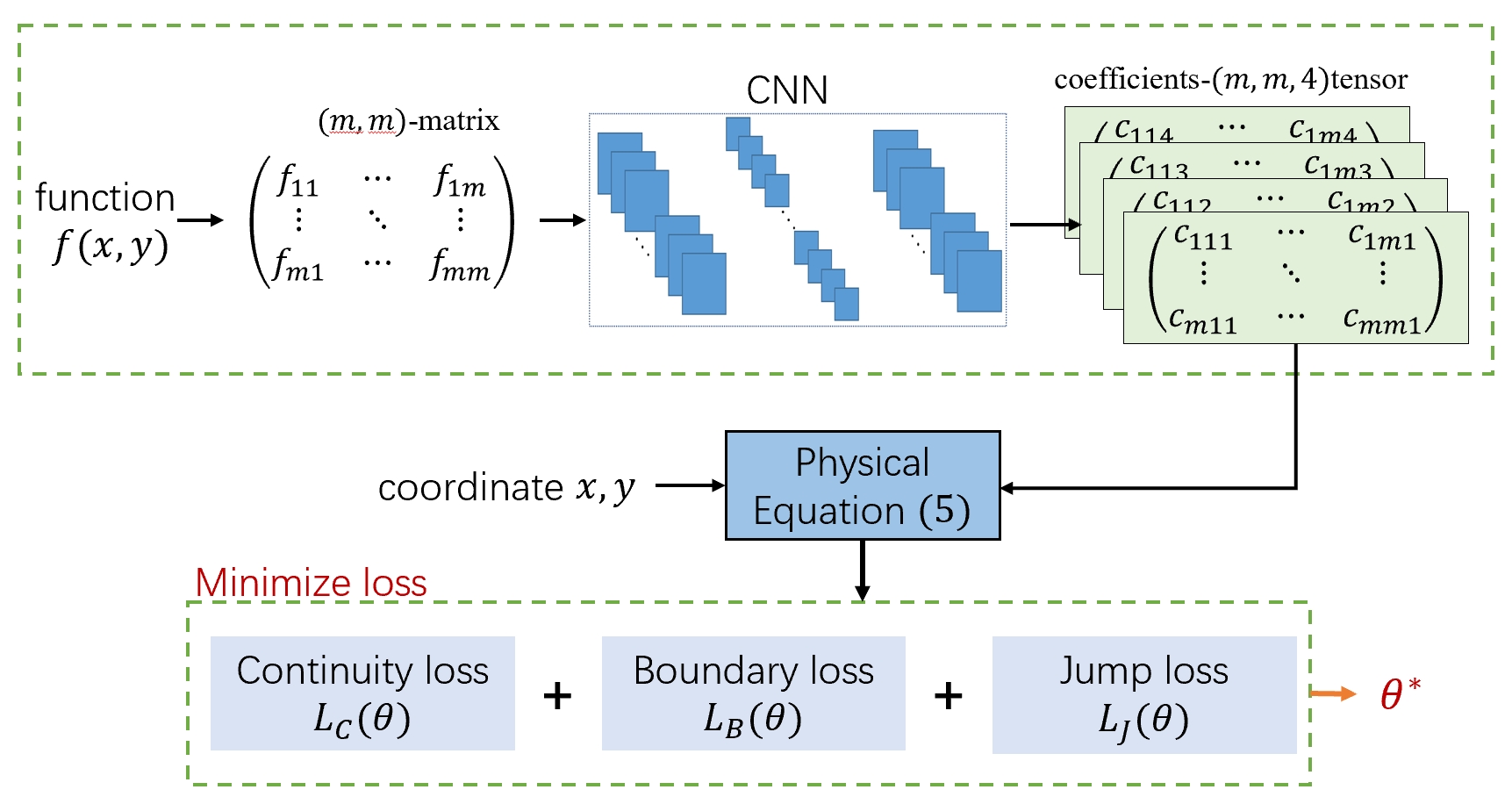}
	\caption{Architecture of PI-TFPNet for two dimensional case.  Here we use the convolutional neural network (CNN) to replace the FNN in one dimensional case.}
	\label{fig_Net2d}
\end{figure*}

\begin{table*}[h]
\setlength\tabcolsep{0.5cm} 
\centering
\begin{tabular}{c|ccc} \hline
 Component & Operations & Input Shape & Output Shape \\\hline
Encoder &  $4\times(\text{Conv2d}+\text{BN}+\text{ReLU})$ & $(1\times16\times16)$ & $(256\times1\times1)$ \\\hline
    \multirow{3}*{Decoder} & $2 \times \begin{bmatrix} \text{ConvTranspose2d}+\text{BN}+\text{ReLU} \\ \text{Conv2d}+\text{BN}+\text{ReLU} \end{bmatrix}$ & \multirow{3}*{$(256\times1\times1)$} & \multirow{3}*{$(16\times16\times4)$} \\
~ & $\text{Conv2d} + \text{BN} + \text{Tanh}$ & ~  \\
~ & Permute & ~ \\\hline
%Trunk Net & $4\times(\text{Linear}+\text{ReLU})$ & $(5)/(10)$ & $(512)$  \\\hline 
\end{tabular}
\caption{Detailed configuration of the CNN (encoder and decoder).}
\label{table:componets}
\end{table*}

\section{Additional Theoretical Results}

We provide the proof of the theorem in the main text. The proof process is essentially the same for both one-dimensional and two-dimensional problems. For simplicity and to avoid repetition, we present it in one dimension.

\subsection{Stability estimate for interface problem}

Firstly, we present the stability estimate for an interface problem with multiple interfaces. In this case, the domain $\Omega$ is partitioned by $\widetilde{\Gamma}=\{x_i\}_{i=0}^M$ into $M$ sub-regions $\{\Delta_i\}_{i=0}^{M-1}$, where each $\Delta_i$ is an interval. The function $u$, defined on $\Omega$, exhibits jumps between adjacent sub-regions. We provide the estimation of $u$ in the following lemma:

\begin{lemma}
\label{lemma1}
    Suppose $f\in L^2(\Omega)$ and $c\in L^{\infty}(\Omega)$. Let $u$ be the solution of 
\begin{equation}
    \left\{\begin{aligned}
    &-u''(x)+c(x)u(x)=f(x),\quad  x\in \Omega/ \widetilde{\Gamma},  \\
    &[u]|_{x_i} = a_i,\quad [u']|_{x_i} = b_i, \quad i\in\{1,...,M-1\},  \\
    &u(x_0) = p_0,\quad u(x_M)=q_0. 
\end{aligned}\right.
\end{equation}
    Then
\begin{equation*}
   \left(\|u\|_{2, \Omega}^*\right)^2\le C\left(\|f\|^2_{0,\Omega}+\sum\limits_{i=1}^{M-1}\left(|a_i|^2+|b_i|^2\right)+p_0^2+q_0^2\right),
\end{equation*}
    where $C$ is a constant independent of $f$, $c$, $a_i$, $b_i$, $p_0$, and $q_0$.
\end{lemma}

\begin{proof}
We decompose $u$ into $u=\sum\limits_{i=1}^M u_i$, where each $u_i$ satisfies the following equations:
\begin{itemize}
    \item For $u_1$:
    \begin{equation*}
\left\{\begin{array}{lr}
-u_1''(x)+c(x)u_1(x)=f(x), & x\in \Omega/ \widetilde{\Gamma}, \\
u_1|_{\partial \Omega}=0, & \\
\left[u_1\right] |_{x_1}=a_1,\quad [u_1'] |_{x_1}=b_1. &
\end{array}\right.
\end{equation*}
\item For $u_i$ where $2 \leq i \leq M-1$:
\begin{equation*}
\left\{\begin{array}{lr}
-u_i''(x)+c(x)u_i(x)=0, & x\in \Omega/ \widetilde{\Gamma}, \\
u_i|_{\partial \Omega}=0, & \\
\left[u_i\right] |_{x_i}=a_i,\quad [u_i'] |_{x_i}=b_i. &
\end{array}\right.
\end{equation*}
\item For $u_M $:
\begin{equation*}
\left\{\begin{array}{lr}
-u_M''(x)+c(x)u_M(x)=0, & x\in \Omega, \\
u_M(x_0)=p_0,\quad u_M(x_M)=q_0. &
\end{array}\right.
\end{equation*}
\end{itemize}

From Lemma 2.1 in \cite{huang2009tailored}, we have:
\begin{equation*}
\left(\|u_1\|_{2, \Omega}^*\right)^2\le C\left(\|f\|^2_{0,\Omega}+|a_1|^2+|b_1|^2\right).\\
\end{equation*}
For $u_i$ where $2 \leq i \leq M-1$:
\begin{equation*}
\left(\|u_i\|_{2, \Omega}^*\right)^2\le C\left(|a_i|^2+|b_i|^2\right).
\end{equation*}
Applying the standard stability estimate for elliptic problems from \cite{evans2022partial}, we obtain:
\begin{equation*}
\left(\|u_M\|_{2, \Omega}^*\right)^2\le C\left(|p_0|^2+|q_0|^2\right).
\end{equation*}
Combining these results yields the desired estimate.
\end{proof}

\subsection{Proof of Theorem 1}

For $x \in \Delta_i$, we have
\begin{equation*}
    u_{\theta}(x) = \alpha_i(\theta) A_1^i(x) + \beta_i(\theta) A_2^i(x) + \int_{x_{i-1}}^{x_i}f(s)G_i(x,s)ds.
\end{equation*}
Given the definition of the local basis functions, $u_{\theta}(x)$ satisfies the following interface problem:
\begin{equation}
\left\{\begin{aligned}
&-\Delta u_{\theta}(x) + c_h(x)u_{\theta}(x) = f(x), \quad x\in \Omega/\widetilde{\Gamma},\\
&[u_{\theta}]|_{x_i} = u_{\theta}(x_i^+)-u_{\theta}(x_i^-),\quad 1\leq i \leq M-1,\\
&[u_{\theta}']|_{x_i} = u'_{\theta}(x_i^+)-u'_{\theta}(x_i^-),\quad 1\leq i \leq M-1, \\
&u_{\theta}(x_0) = u_{\theta}(x_0),\qquad  u_{\theta}(x_N) = u_{\theta}(x_N). 
\end{aligned}\right.
\end{equation}
Here, $c_h(x)$ is the piece-wise linear approximation of $c(x)$. 

Subtracting the above equation from Eq.~(3) in the main text gives the following interface problem for the error function $e(x)$:
\begin{equation}
\left\{\begin{aligned}
&-\Delta e(x) + c_h(x)e(x) = R_h(x), &x\in \Omega/ \Gamma,\\
&[e]|_{x} = [u_{\theta}]|_{x},\ [e']|_{x} = [u_{\theta}']|_{x},&x\in \widetilde{\Gamma}/ (\Gamma\cup \partial\Omega),\\
&[e]|_{x} = [u_{\theta}]|_{x}-g_D(x), &x\in \Gamma,\\
&[e']|_{x} = [u_{\theta}']|_{x}-g_N(x),&x\in \Gamma,\\
&e(x) = u_{\theta}(x), &x\in\partial\Omega. 
\end{aligned}\right.
\end{equation}
Here, $R_h(x) = (c(x) - c_h(x)) u(x)$. Since $|c(x) - c_h(x)| \leq Ch^2$, it follows that $\|R_h\|_{0, \Omega} \leq Ch^2 \|u\|^*_{0, \Omega}$.

By applying Lemma 1, we have:
\begin{equation*}
    \|u\|^*_{0, \Omega} \leq C \left(\|f\|_{0, \Omega} + \|g_D\|_{\infty, \Gamma} + \|g_N\|_{\infty, \Gamma}\right),
\end{equation*}
which implies:
\begin{equation*}
    \|R_h\|_{0, \Omega} \leq Ch^2 \left(\|f\|_{0, \Omega} + \|g_D\|_{\infty, \Gamma} + \|g_N\|_{\infty, \Gamma}\right).
\end{equation*}

Using Lemma 1 again for the error function $e(x)$, and combining this with the loss function from Eq.~(8) in the main text, we obtain Theorem 1.

\subsection{Stability Estimate for Singular Perturbation Interface Problem}

Next, we provide the stability estimate for a singular perturbation interface problem involving multiple interfaces.

\begin{lemma}
\label{lemma2}
    Suppose $f\in L^2(\Omega)$ and $c\in L^{\infty}(\Omega)$. Let $u$ be the solution of 
\begin{equation}
    \left\{\begin{aligned}
    &-\varepsilon u''(x)+c(x)u(x)=f(x),\quad  x\in \Omega/ \widetilde{\Gamma},  \\
    &[u]|_{x_i} = a_i,\quad [\varepsilon u']|_{x_i} = b_i, \quad i\in\{1,...,M-1\},  \\
    &u(x_0) = p_0,\quad u(x_M)=q_0. 
\end{aligned}\right.
\end{equation}
    Then
\begin{equation*}
   \left(\|u\|_{\varepsilon, \Omega}^*\right)^2\le C\left(\|f\|^2_{0,\Omega}+\sum\limits_{i=1}^{M-1}\left(|a_i|^2+|b_i|^2\right)+p_0^2+q_0^2\right),
\end{equation*}
    where $C$ is a constant independent of $f$, $c$, $a_i$, $b_i$, $p_0$, and $q_0$.
\end{lemma}

\begin{proof}
We decompose $u$ into $u=\sum\limits_{i=1}^M u_i$, where each $u_i$ satisfies the following equations:
\begin{itemize}
    \item For $u_1$:
\begin{equation*}
\left\{\begin{array}{lr}
-\varepsilon u_1''(x)+c(x)u_1(x)=f(x), & x\in \Omega/ \widetilde{\Gamma}, \\
u_1|_{\partial \Omega}=0, & \\
\left[u_1\right] |_{x_1}=a_1,\quad [\varepsilon u_1'] |_{x_1}=b_1. &
\end{array}\right.
\end{equation*}
\item For $u_i$ where $2 \leq i \leq M-1$:
\begin{equation*}
\left\{\begin{array}{lr}
-\varepsilon u_i''(x)+c(x)u_i(x)=0, & x\in \Omega/ \widetilde{\Gamma}, \\
u_i|_{\partial \Omega}=0, & \\
\left[u_i\right] |_{x_i}=a_i,\quad [\varepsilon u_i'] |_{x_i}=b_i. &
\end{array}\right.
\end{equation*}
\item For $u_M $:
\begin{equation*}
\left\{\begin{array}{lr}
-\varepsilon u_M''(x)+c(x)u_M(x)=0, & x\in \Omega, \\
u_M(x_0)=p_0,\quad u_M(x_M)=q_0. &
\end{array}\right.
\end{equation*}
\end{itemize}

From the proof of Theorem 3.3 in \cite{huang2009tailored} (Eq.~(50)), we obtain:
\begin{equation*}
\left(\|u_1\|_{\varepsilon, \Omega}^*\right)^2\le C\left(\|f\|^2_{0,D}+|a_1|^2+|b_1|^2\right).
\end{equation*}
For $u_i$ where $2 \leq i \leq M-1$:
\begin{equation*}
\left(\|u_i\|_{\varepsilon, \Omega}^*\right)^2\le C\left(|a_i|^2+|b_i|^2\right).
\end{equation*}
For $u_M$, which does not face discontinuous interfaces, we multiply its equation by $u_M$ and integrate over $\Omega$. Applying the Cauchy-Schwartz inequality and Young’s inequality, we obtain:
\begin{equation*}
\left(\|u_M\|_{\varepsilon, \Omega}^*\right)^2\le C\left(|p_0|^2+|q_0|^2\right).
\end{equation*}
Combining these results yields the desired estimate.
\end{proof}

\subsection{Proof of Theorem 2}
The proof of Theorem 2 is very similar to the proof of Theorem 1, in which we replace Lemma 1 by Lemma 2. So we omit it here for simplicity.

\section{Additional Experimental Results}

Next, we provide supplementary experimental results, including the decay of the loss function during training and the relative $L^2$ norm on the validation set. Figures \ref{fig:1d smooth}, \ref{fig:1d singular}, \ref{fig:1d high-contrast-appendix}, \ref{fig:2d interface}, and \ref{fig:2d singular} correspond to the 1D smooth, 1D singular, 1D high-contrast, 2D smooth, and 2D singular cases discussed in the main text, respectively.

\begin{figure}[H]
	\centering
	\subfloat[]{\includegraphics[width=.23\columnwidth]{fig_PI-TFPONet/1d_smooth_example.png}}
	\subfloat[]{\includegraphics[width=.23\columnwidth]{fig_PI-TFPONet/1d_smooth_errors.png}}
        \subfloat[]{\includegraphics[width=.23\columnwidth]{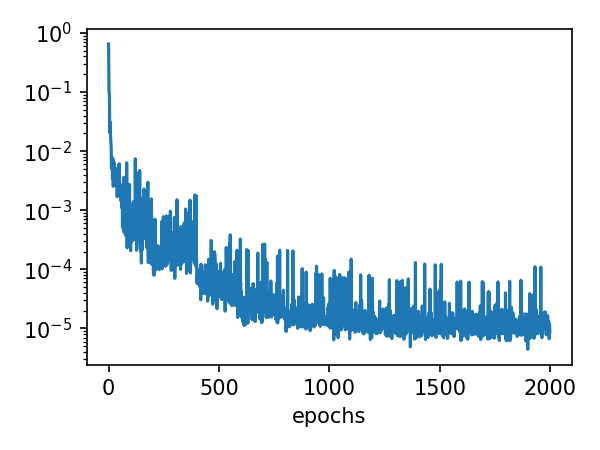}}
        \subfloat[]{\includegraphics[width=.23\columnwidth]{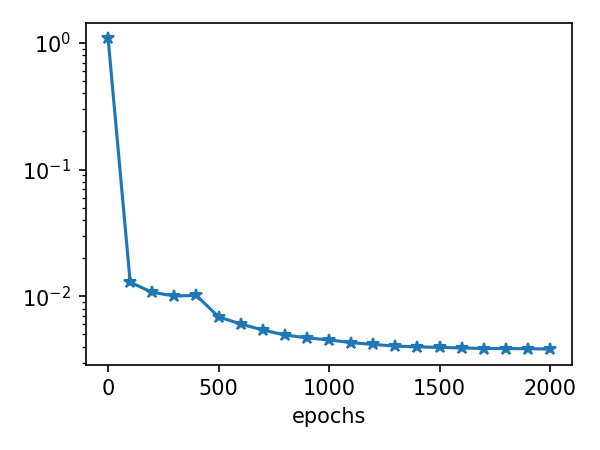}}
	\caption{[1d smooth]: (a)PI-TFPONet’s refinement predicted solution. (b)Error distribution across 200 test examples. Solid line: median error, shaded area: min to max error range. (c)PI-TFPONet’s training loss curve. (d)PI-TFPONet’s relative $L^2$ error on the validation set.}\label{fig:1d smooth}
\end{figure}

\begin{figure}[!h]
	\centering
	\subfloat[]{\includegraphics[width=.23\columnwidth]{fig_PI-TFPONet/1d_singular_example.png}}
	\subfloat[]{\includegraphics[width=.23\columnwidth]{fig_PI-TFPONet/1d_singular_errors.png}}
        \subfloat[]{\includegraphics[width=.23\columnwidth]{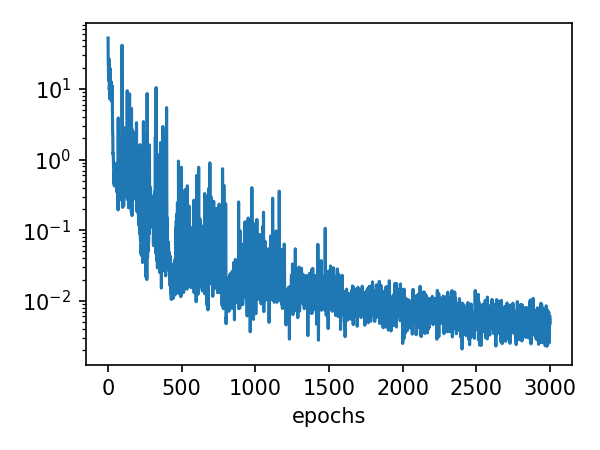}}
        \subfloat[]{\includegraphics[width=.23\columnwidth]{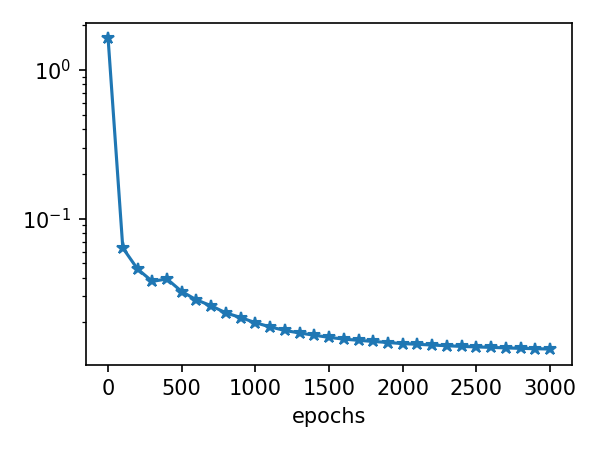}}
	\caption{[1d singular]: (a)PI-TFPONet’s refinement predicted solution. (b)Error distribution across 200 test examples. Solid line: median error, shaded area: min to max error range. (c)PI-TFPONet’s training loss curve. (d)PI-TFPONet’s relative $L^2$ error on the validation set.}\label{fig:1d singular}
\end{figure}
\begin{figure}[!h]
	\centering
	\subfloat[]{\includegraphics[width=.23\columnwidth]{fig_PI-TFPONet/1d_high_contrast_example.png}}
	\subfloat[]{\includegraphics[width=.23\columnwidth]{fig_PI-TFPONet/1d_high_contrast_errors.png}}
        \subfloat[]{\includegraphics[width=.23\columnwidth]{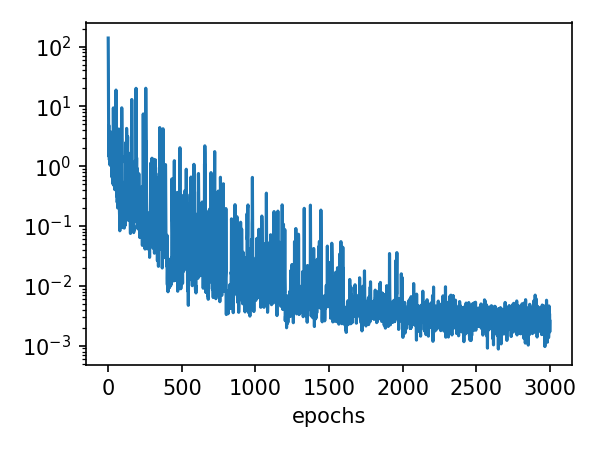}}
        \subfloat[]{\includegraphics[width=.23\columnwidth]{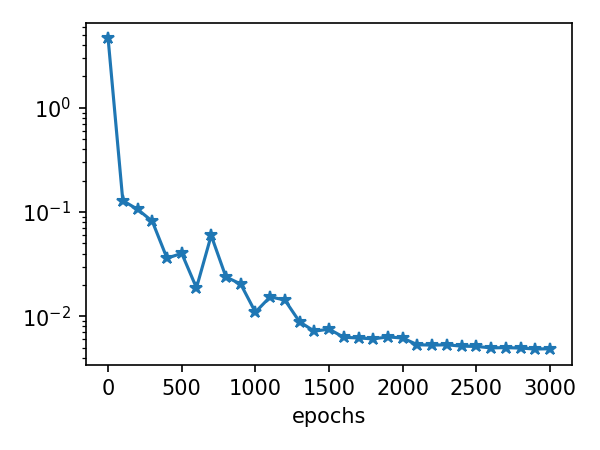}}
	\caption{[1d high-contrast]: (a)PI-TFPONet’s refinement predicted solution. (b)Error distribution across 200 test examples. Solid line: median error, shaded area: min to max error range. (c)PI-TFPONet’s training loss curve. (d)PI-TFPONet’s relative $L^2$ error on the validation set.}\label{fig:1d high-contrast-appendix}
\end{figure}
\begin{figure}[!h]
	\centering
	\begin{minipage}{0.31\textwidth}
		\subfloat[]{\includegraphics[width=\textwidth]{fig_PI-TFPONet/2d_refine.png}}
	\end{minipage}\quad	
	\begin{minipage}{0.3\textwidth}
		\subfloat[]{\includegraphics[width=\textwidth]{fig_PI-TFPONet/2d_ground.png}}\\
		\subfloat[]{\includegraphics[width=\textwidth]{fig_PI-TFPONet/2d_error.png}}
	\end{minipage}\quad
	\begin{minipage}{0.3\textwidth}
		\subfloat[]{\includegraphics[width=\textwidth]{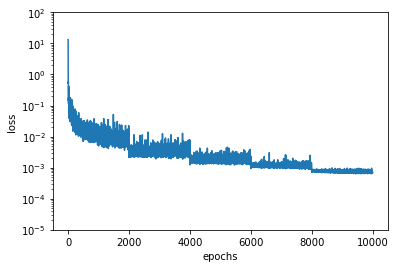}}\\
		\subfloat[]{\includegraphics[width=\textwidth]{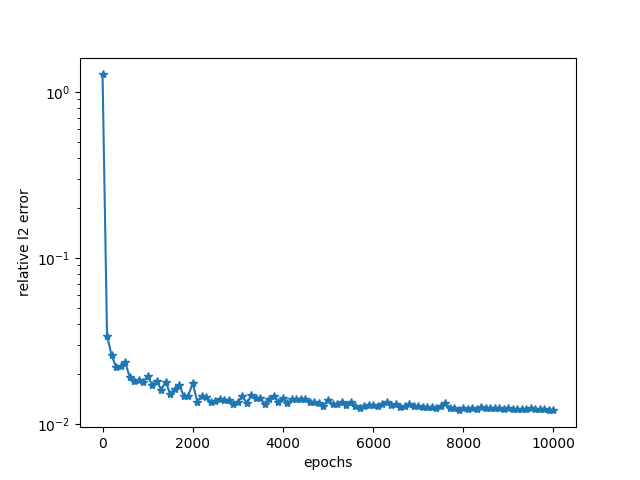}}
	\end{minipage}
	\caption{[2d interface]: (a)PI-TFPONet's refinement predicted solution. (b)ground truth solution. (c)error distribution over domain. (d)PI-TFPONet's training loss curve. (e)PI-TFPONet's relative $L^2$ error on the validation set.}\label{fig:2d interface}
\end{figure}

\newpage
\begin{figure}[H]
	\centering
	\begin{minipage}{0.31\textwidth}
		\subfloat[]{\includegraphics[width=\textwidth]{fig_PI-TFPONet/2d_singular_refine.png}}
	\end{minipage}\quad	
	\begin{minipage}{0.3\textwidth}
		\subfloat[]{\includegraphics[width=\textwidth]{fig_PI-TFPONet/2d_singular_ground.png}}\\
		\subfloat[]{\includegraphics[width=\textwidth]{fig_PI-TFPONet/2d_singular_error.png}}
	\end{minipage}\quad
	\begin{minipage}{0.3\textwidth}
		\subfloat[]{\includegraphics[width=\textwidth]{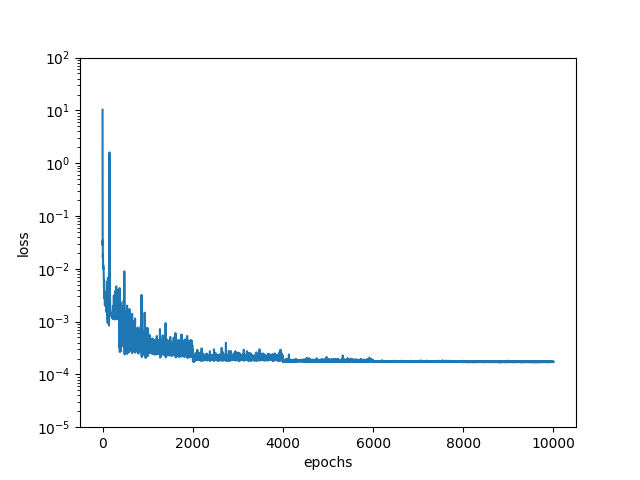}}\\
		\subfloat[]{\includegraphics[width=\textwidth]{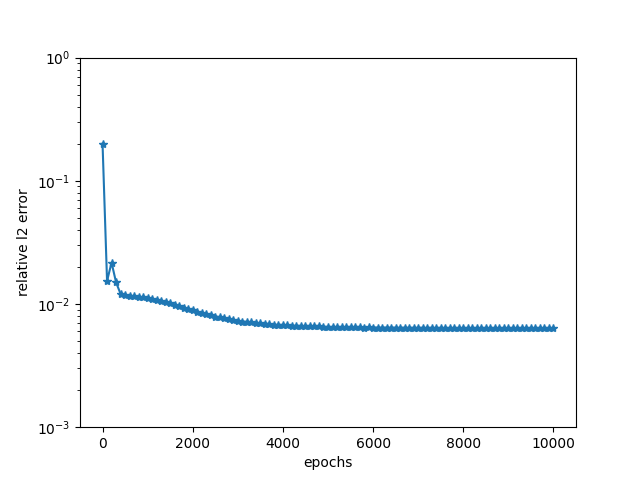}}
	\end{minipage}
	\caption{[2d singular]: (a)PI-TFPONet's refinement predicted solution. (b)ground truth solution. (c)error distribution over domain. (d)PI-TFPONet's training loss curve. (e)PI-TFPONet's relative $L^2$ error on the validation set.}\label{fig:2d singular}
\end{figure}

Finally, we present the relative $L^\infty$ error of our model and the baseline models on the test set in Table \ref{tab_results_appendix} to further demonstrate the effectiveness of our model. All models were trained using an NVIDIA GeForce RTX 3060 GPU. For interface problems with additional singularities, including the 1D singular, 1D high-contrast, and 2D singular cases, our model achieves the highest accuracy, surpassing even those trained with supervised data. For general interface problems, our model is slightly less accurate than IONet but outperforms other models. Considering that our model does not require additional supervised data and uses only a single network, this level of accuracy is reasonable.

\begin{table}[!h]
\centering
\scalebox{0.99}{
\begin{tabular}{ccccccc}
\toprule
Supervision & Method & 1d smooth  & 1d singular   & 1d high-contrast  & 2d interface & 2d singular\\ \midrule
supervised & DeepONet &  6.19e-02  & 1.04e-01   &  6.33e-02 &  2.59e-02 & 3.12e-01\\ 
supervised & IONet &   \textbf{3.62e-03} &  5.51e-02 &  2.74e-02 & \textbf{6.00e-03} & 1.55e-01 \\ 
unsupervised & PI-DeepONet &  8.90e-01  & 1.36e-00  & 1.03e-00  & 7.21e-01 & 5.28e-01 \\ 
unsupervised & PI-IONet &  8.56e-03  & 8.95e-01   &  3.80e-01 &  6.42e-02 & 3.92e-01\\ \midrule
unsupervised & PI-TFPONet & 7.31e-03   & \textbf{1.40e-02}   & \textbf{4.23e-03}  & 8.42e-03 &\textbf{6.86e-02}  \\ \bottomrule
\end{tabular}}
\caption{A comparison of the relative $L^{\infty}$ error for different interface problems.}
\label{tab_results_appendix}
\end{table}

\end{document}